\documentclass[a4paper]{article}
\usepackage[utf8]{inputenc}
\usepackage{graphicx}
\usepackage{amsmath}
\usepackage{amssymb}
\usepackage{amsfonts}
\usepackage{amsthm}
\usepackage{graphicx}
\usepackage{makeidx}
\usepackage{tikz}
\usetikzlibrary{shapes,snakes,shapes.geometric}
\usetikzlibrary{matrix,decorations.pathreplacing}
\usepackage{mathrsfs}
\usepackage[colorinlistoftodos,french]{todonotes}
\usepackage{comment}
\usepackage{hyperref}

\DeclareMathOperator{\gp}{GP}
\DeclareMathOperator{\kgp}{\mathbb K^{GP}}
\DeclareMathOperator{\loi}{\overset{d}{=}}

\DeclareMathOperator{\pp}{P}
\DeclareMathOperator{\dd}{d}
\DeclareMathOperator{\dgp}{d_{GP}}
\DeclareMathOperator{\supp}{supp}
\DeclareMathOperator{\exc}{exc}
\DeclareMathOperator{\dir}{Dir}
\DeclareMathOperator{\poidir}{PD}

\DeclareMathOperator{\mittagl}{ML}
\DeclareMathOperator{\cepsalpha}{\mathcal C_{\varepsilon^{1/\alpha}}}
\DeclareMathOperator{\epsalpha}{\varepsilon^{1/\alpha}}
\DeclareMathOperator{\cepsalphaf}{\mathcal C_{\varepsilon^{1/\alpha}}^\uparrow}
\DeclareMathOperator{\dimh}{\dim_{\mathrm{H}}}
\DeclareMathOperator{\dimms}{\overline{\dim_{\mathrm M}}}
\DeclareMathOperator{\dimmi}{\underline{\dim_{\mathrm M}}}
\DeclareMathOperator{\dimm}{\dim_{\mathrm M}}
\DeclareMathOperator{\nbigeps}{N_\varepsilon^{\textrm{b}}}

\DeclareMathOperator{\lbeta}{Beta}

\newcounter{HypN}

\newenvironment{EqHypothese}
  {\stepcounter{HypN}%
    \addtocounter{equation}{-1}%
    \equation}
  {\endequation}

\newtheorem{lthe}{Theorem}
\newtheorem{lpro}{Proposition}
\newtheorem{llem}{Lemma}
\newtheorem{lcor}{Corollary}

\theoremstyle{definition}

\theoremstyle{remark}
\newtheorem{lrem}{Remark}

\newcommand{\bbE}{\mathbb E}
\newcommand{\bbP}{\mathbb P}
\newcommand{\Xexc}{X ^{\exc}}
\newcommand{\smk}{s_k^{(m)}}
\newcommand{\dup}{\overline{d}}
\newcommand{\dlow}{\underline{d}}
\newcommand{\zb}{\overline{\zeta}}
\newcommand{\munp}[1]{{\mathcal M}_1^{#1}}

\newcommand{\old}{\color{red}}

\title{Self-similar random structures defined as fixed points of distributional equations}
\author{Lucas Iziquel}
\date{\today}

\begin{document}

\maketitle

\begin{abstract}
We consider fixed-point equations for probability distributions on isometry classes of measured metric spaces. The construction is required to be recursive and tree-like, but we allow loops for the geodesics between points in the support of the measure: one can think of a $\beta$-stable looptree decomposed around one loop. We study existence and uniqueness of solutions together with the attractiveness of the fixed-points by iterating. We obtain bounds on the Hausdorff and upper Minkowski dimension, which appear to be tight for the studied models. This setup applies to formerly studied structures as the $\beta$-stable trees and looptrees, of which we give a new characterization and recover the fractal dimensions.
\end{abstract}

\section{Introduction}

The Brownian Continuum Random Tree, or Brownian CRT, was first introduced by Aldous in \cite{aldous1991continuum}, and further studied in \cite{aldous1993continuum}. In the original setting, it arises as the scaling limit of uniformly random labelled trees, as well as the one of rescaled Galton-Watson trees with finite offspring variance. Since then, similar objects that arise as scaling limits of discrete structures have been studied. For instance, we can think of the $\beta$-stable trees for $1<\beta\leq 2$ studied by Duquesne and Le Gall in \cite{duquesne2002random,duquesne2005probabilistic} which are  particular cases of Lévy trees introduced by Le Gall and Le Jan in \cite{leGallleJan1998}. $\beta$-stable trees are the scaling limits of rescaled Galton-Watson trees whose offspring distribution is in the domain of attraction of a stable law, and generalize the Brownian CRT which corresponds to the $2$-stable tree. Another example is the stable looptrees first introduced by Curien and Kortchemski in \cite{curien2014random} which are the scaling limits of discrete looptrees.

These objects are random $\mathbb R$-trees or continuum trees, or at least tree-like if we want to include the stable looptrees which almost surely contain loops. They enjoy natural recursive decompositions as they appear when studying branching processes and fragmentation processes. A formal setup to deal with these objects is to define them as random variables in a set of equivalence classes of metric spaces; spaces also endowed with a probability measure which will allow us to sample points at random, and on which we eventually fix a root point.

We will focus on random objects that verify self-similarity properties: decomposing the entire space by removing some branchpoints will leave connected components which, after rescaling, are distributed as independent copies of the entire space conditionally on the removed branchpoints.
This is the case for the Brownian CRT (\cite{aldous1994recursive}) and for the $\beta$-stable trees (\cite{Marchal2008note}) but not for a general Lévy tree. More precisely, these self-similarity properties can take the form of a \emph{fixed-point equation} when we look at the distribution of the objects. The general setup presented with more details in Section \ref{sectiondescriptionmodele} can be outlined as follows: for a random variable $X$ taking values in the space $\mathbb S$ of "objects", which we want to be Polish, we study equations of the form:
\begin{equation} \label{distributionalfixedpointequation} 
 X \loi T \left( \left( X_i \right)_{i \geq 0} , \Xi \right),
\end{equation}
where $X_i$ are independent and identically distributed random variables with the same distribution as $X$, $T$ is a suitable map, and $\Xi$ represents an external source of randomness.

The first ones to consider this point of view for $\mathbb R$-trees are Albenque and Goldschmidt who proved in \cite{albenque2015brownian} that the Brownian CRT is the unique fixed-point of some natural decomposition first described by Aldous in \cite{aldous1994recursive}. This result has been generalized to construct a larger class of random real trees by Broutin and Sulzbach in \cite{broutin2016self}, from which we borrow the notations. Note also that the study of the $\beta$-stable trees by Rembart and Winkel in \cite{MR3846837} and also in \cite{chee2018recursive} with Chee fall in this general setting.
The fact that these classical objects verify a \emph{distributional fixed-point equation} pushes us to wonder about general properties of such an equation or their fixed-points. The first to come to mind when studying Equation \eqref{distributionalfixedpointequation} is whether there exists a fixed-point to this equation, and if it is unique at least in distribution. A subsequent natural question to ask is if the fixed-point can be recovered by iterating the appplication $T$ in Equation \eqref{distributionalfixedpointequation}. This is addressed in \cite{broutin2016self,albenque2015brownian,MR3846837} in different cases, and strongly depend on the Polish space $\mathbb S$ in which is considered the problem.

Given the existence of such fixed-points, one may try to deduce their geometric properties such as fractal dimensions, directly from the equation.
Indeed, this is a natural observation concerning classical deterministic fractals such as the curves of von Koch, or Cantor sets on the real line or in the plane. We refer to Falconer (\cite{falconer1990fractal}) for more details. It has been made rigorous in \cite{broutin2016self} in which both the almost sure Hausdorff dimension and upper Minkowsi dimension of the fixed-point appear as a parameter of the application $T$ under reasonable hypotheses. The Hausdorff dimension of the Brownian CRT, which was shown to be almost surely equal to $2$ in  \cite{duquesne2005probabilistic} was recovered this way. Their framework also extends to recursive triangulations of the disk (see \cite{broutin2015dual} and \cite{curien2011random}).
A similar result is stated in \cite{MR3846837}, giving the almost sure Hausdorff dimension of their fixed-points which applies for instance to the case of the $\beta$-stable trees with $1<\beta\leq2$.

This article aims at giving a general toolbox to study a \emph{distributional fixed-point equation} with specific parameters, which aims to be systematic. Respecting a tree structure, the rescaled copies are "glued" onto a random metric space. In \cite{broutin2016self}, this metric space is a finite collection of points and one only glues together a fixed finite number of copies. Here we consider the case when this space actually contains length with positive probability, and the fixed number of subspaces we concatenate for a given \emph{distributional fixed-point equation} will either be finite or countable. See Figure \ref{fig:intro} for an illustration.

\begin{figure}[!ht]
\label{fig:intro}
\includegraphics[scale=.3]{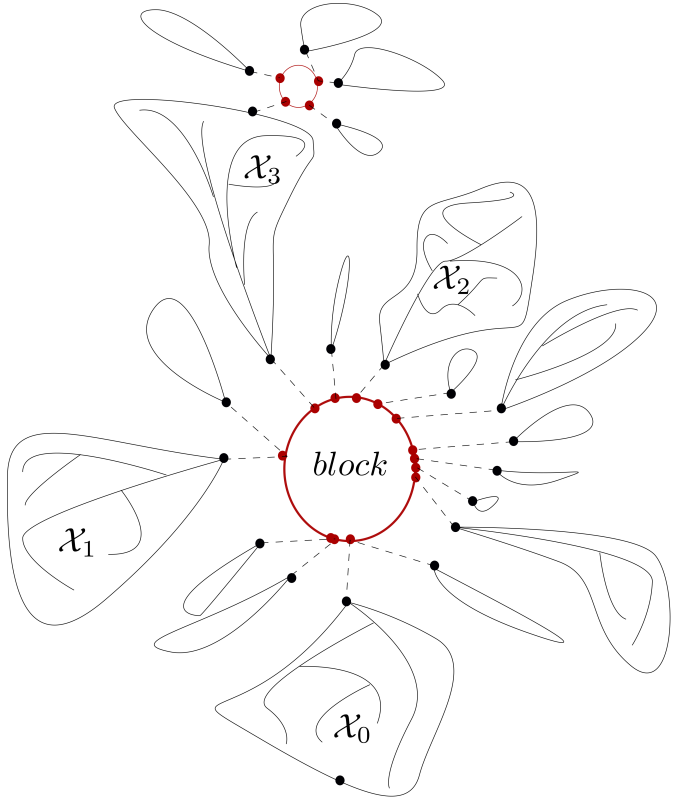}
\caption{ An example of gluing in our setting: in black are represented the rescaled copies of the input space $\mathcal X_i$, which are glued onto other metric spaces represented by the coloured circles and indicated as "blocks". Points linked by a dashed line are identified. }
\end{figure}

Our framework covers a number of examples of well-known random metric spaces:
\begin{itemize}
\item The $\beta$-stable trees for $1<\beta \leq 2$ from its spinal mass decomposition given by Haas, Pitman and Winkel in \cite{haas2009spinal}, the $2$-stable tree being the Brownian CRT,
\item The $\beta$-stable Looptrees from \cite{curien2014random} via a decomposition around one loop.
\end{itemize}

Because we want our theorems to apply not only to $\mathbb R$-trees but to a more general class of tree-like objects which include the stable looptrees, we need our proofs not to rely on the study of random real-valued functions encoding the spaces but directly on the spaces themselves. As a brief reminder, when considering $\mathbb R$-trees, a handy description of these spaces relies on their relation with  random functions which can be studied by powerful tools. See for instance \cite{evans2007probability,le2005random}:  $\mathbb R$-trees are often defined as $[0,1]$ endowed with a distance depending on a random function. Many of the results concerning random $\mathbb R$-trees in the litterature are deduced from the study of the random process coding the tree: our proofs are clear from this description and deal directly with the properties of the spaces themselves inside our Polish space $\mathbb S$.

The paper is organized as follows. In Section \ref{section2preliminaries}, after introducing our notations, we give the necessary background about the objects we consider, which concerns mostly Hausdorff-like metrics and the corresponding spaces; then we describe the class of maps $T$ we consider, and the important parameters to look at such as the height of a random point; lastly we recall definitions and results about geometric properties of metric spaces. Our main results are stated in Section \ref{sectionresults}, followed by applications to previously studied objects and a brief description of the main techniques. Section \ref{sectionexuniq} contains the proof of the first theorems, concerning existence, uniqueness and convergence of iterated processes towards the possible solutions. Section \ref{sectiongeo} is devoted to the proofs of the geometric properties of the fixed-points. Section \ref{sec:appliproof} regroups proofs concerning applications.

\newpage

\setcounter{tocdepth}{2}
\tableofcontents

\newpage

\section{Preliminaries}
\label{section2preliminaries}
\subsection{Rooted plane trees and notations}
Let $\mathcal{U}:=\bigcup\limits_{n=0}^\infty (\mathbb Z^+)^n$ be the set of finite words on the set of positive integers $\mathbb Z^+:=\{1,2,\ldots \}$. For two words $u=(u_1, \ldots , u_n)$ and $v=(v_1 , \ldots , v_m)$ both in $\mathcal U$, let $uv$ be the concatenation of $u$ and $v$ defined as $(u_1 , \ldots , u_n ,v_1 , \ldots , v_m)$.

A rooted plane tree $t$ is a subset of  $\mathcal{U}$ such that $\emptyset$ is in $t$, if $v=uk$ is in $t$ with $k \in \mathbb Z^+$ and $u \in \mathcal{U}$ then $u$ is in $t$, and if $uk \in t$ for $(u,k) \in \mathcal{U} \times \mathbb Z^+$, then $uj \in t$ for all $ j \leq k$.

On a tree $t$, for $i$ a vertex in $t$ we consider $t_i$ the subtree of $t$ rooted in $i$: it contains $i$ and all the vertices with prefix $i$ which are in $t$. Let $\kappa_i$ denote the set of direct children of $i$, meaning every node $ik$ for $k \in \mathbb Z^+$ such that $ik \in t$, this set is empty whenever $i$ is a leaf. If $i=(i_1, \ldots , i_n)$, we call $n$ the generation or height of $i$ and write $\vert i \vert :=n$. We also define the father of $i$ in $t$ being simply $i-:=i_1 \ldots i_{n-1}$, where only $\emptyset$ has no father. \label{notasubtree}

For $j$ another vertex in $t$, let $q(i,j)$ be the most recent common ancestor of $i$ and $j$, which is also their largest common prefix. We denote by $\pi(i,j)$ the set of vertices in the only path between $i$ and $j$ in $t$ including both ends $i$ and $j$, to which we then remove $q(i,j)$. \label{notapath}

\subsection{Metrics and convergence}

\subsubsection{$\gp$-isometry classes}

We refer in this section to the basic setup from \cite{broutin2016self}, and to \cite{evans2007probability,gromov2007metric,greven2009convergence} for more details and proofs.

We will focus on rooted measured metric spaces $\mathfrak X= \left( X,d,\mu,\rho\right)$. Such a space is a complete separated metric space $\left( X,d \right)$ endowed with a probability measure $\mu$, and a distinguished point $\rho$ in $X$. Between two rooted measured metric spaces $\mathfrak X = \left( X,d,\mu,\rho\right)$ and $ \mathfrak X' =  \left( X',d',\mu',\rho' \right)$ we define the \emph{Gromov-Prokhorov} distance by
\[ \dgp \left(\mathfrak X, \mathfrak X' \right) := \inf_{ Z , \phi , \phi ' } \left\{ d^Z \left( \phi ( \rho ) , \phi ' ( \rho ' ) \right) + \dd_{\pp} ^Z \left( \phi_* ( \mu ) , \phi_* '( \mu ') \right) \right\},\]
where the infimum is taken over all metric spaces $(Z,d^Z)$, and isometries $\phi : X \rightarrow Z$ and $\phi' : X' \rightarrow Z$. Here $\phi_* (\mu)$ is the push-forward of $\mu$, and $\dd_{\pp} ^Z$ denotes the \emph{Prokhorov} metric on the set of probability measures on $Z$, that is for all $\nu, \delta$ probability measures on $Z$:
\begin{multline*}
\dd_{\pp} ^Z \left( \nu,\delta \right) := \inf \left\{ \varepsilon>0 \: ; \: \nu(A) \leq \delta(A^\varepsilon) + \varepsilon \mbox{ and } \delta(A) \leq \nu(A^\varepsilon) + \varepsilon \mbox{ for all} \right. \\ \left. \mbox{measurable sets } A \right\}
\end{multline*}
with $A^\varepsilon:= \{ x \in Z : d^Z(x,A) \leq \varepsilon \}$.

We say that two rooted measured metric spaces are $\mathfrak X$ and $\mathfrak X'$ $\gp$-isometric if $\dgp (\mathfrak X, \mathfrak X')=0$, which is equivalent to the existence of an isometry $\phi$ from $\supp(\mu)$ onto $\supp(\mu')$ ($\supp(\mu)$ denotes the support of $\mu$) such that $\phi(\rho)=\rho'$ and $\phi_*(\mu)=\mu'$. The pseudometric $\dgp$ induces a metric on the set $\kgp$ of $\gp$-isometry classes of rooted measured metric spaces and turns this set into a Polish space \cite{greven2009convergence}. Hence we can and will consider some random variables $\mathcal X$ taking values in this set.

\begin{lrem}
Spaces inside a $\gp$-isometry class, are not necessarily all isometric. For instance, one can think of the triangle with mass $\frac{1}{3}$ at each node, which is in the same $\gp$-isometry class as the unrooteed tree with $3$ leaves each with mass $\frac{1}{3}$ and one internal node. We also stay in the same class if we add massless portions outside of the support of the measure, each by only one link so that no chord is created between points of the support.
\end{lrem}

For a rooted measured metric space $\mathfrak X=\left(X, d,\mu,\rho\right)$, we denote by $h(\mathfrak{X}):= \sup \left\{d (\rho,x) : x \in \supp(\mu)\right\} \in \mathbb R ^+ \cup \{+\infty\}$ its essential height. Since this is constant on $\gp$-isometry classes, this extends naturally to a class function, defined on $\gp$-isometry classes.

\subsubsection{Distance matrix distribution}
\label{matrixdist}

We give another description of $\gp$-isometry classes of rooted measured metric spaces. It follows the general idea of Gromov in \cite{gromov2007metric} Section $3.\frac{1}{2}$ and the setup from Greven, Pfaffelhuber and Winter in \cite{greven2009convergence}. It is extended to random variables with values in $\kgp$.

For $k \in \mathbb N$, let $\mathbb M_k:= \{(m_{ij})_{0 \leq i,j \leq k} ; \forall i,j \: m_{ij}=m_{ji} \geq 0,m_{ii}=0\}$ be the set of symmetric $(k+1)$ by $(k+1)$ matrices with non-negative entries and zeroes on the diagonal. We also write $\mathbb M_{\mathbb N}$ for the set of infinite dimensional matrices satisfying these properties. For a rooted measured metric space $\mathfrak X = \left( X,d,\mu , \rho \right)$, let $\zeta_1 , \zeta_2 , \ldots$ be independent points with distribution $\mu$ on $X$, and set $\zeta_0:= \rho$. The random matrix $(d(\zeta_i,\zeta_j))_{i,j \geq 0}$ is in $\mathbb M_{\mathbb N}$, and its distribution does not depend on the representative of the $\gp$-isometry class of $\mathfrak X$. Hence, we can define the \emph{distance matrix distribution} $\nu^{\mathfrak X} \in \mathcal  M_1(\mathbb  M_{\mathbb N})$ (the set of probability measures on $\mathbb M_{\mathbb N}$) for any element $\mathfrak X$ in $\kgp$. For a probability distribution $\mathbf P$ on $\kgp$, we also define the probability measure
\begin{equation}
\label{nudelta}
\nu^{\mathbf P} (A) := \int \nu^{\mathfrak X} (A) \mathrm d \mathbf P(\mathfrak X) , \quad A \subset \mathbb M_{\mathbb N} \mbox{ measurable.}
\end{equation}

The importance of $\nu^{\mathfrak X}$ resides in Gromov's reconstruction theorem which we adapt to our notations:
\begin{lpro}{(\cite{gromov2007metric} Section $3.\frac{1}{2}$)}
For $\mathfrak X, \mathfrak U \in \kgp$, we have $\nu^{\mathfrak X} =\nu^{\mathfrak U} $ if and only if $\mathfrak X = \mathfrak U$.
\end{lpro}

The counterpart concerning probability measures can be stated as follows:
\begin{lpro}{(Corollary 3.1 in \cite{greven2009convergence})}
\label{critereCVmesures}
We have the following results:
\begin{enumerate}
\item[(i)] For $\mathbf P_1,\mathbf P_2 \in \mathcal M_1(\kgp)$, we have $\mathbf P_1=\mathbf P_2$ if and only if $\nu^{\mathbf P_1}=\nu^{\mathbf P_2}$.
\item[(ii)] For probability distributions $\mathbf P_n$, for $n \geq 1$, on $\kgp$, the sequence $(\mathbf P_n)_{n \in \mathbb N}$ converges weakly if and only if $(\nu^{\mathbf P_n})_{n \in \mathbb N}$ converges weakly and the family $(\mathbf P_n)_n$ is tight.
\end{enumerate}
\end{lpro}

For a random variable $\mathcal X$ taking values in $\kgp$, we will also use the notation $\bbE [ \nu^{\mathcal X}]$ defined as $\nu^{\mathfrak L(\mathcal X)}$ (where $\mathfrak L(\mathcal X)$ stands for the distribution of $\mathcal X$) or more precisely as follows on every measurable set $A$:
\begin{equation*}
\bbE [\nu^{\mathcal X}](A):=\bbE [\nu^{\mathcal X}(A)].
\end{equation*}
Proposition \ref{critereCVmesures} allows us to compare the distance matrix distributions of random variables with values in $\kgp$, rather than computing directly the $\gp$ distance.

Most of the results from \cite{buragoburagoivanov,greven2009convergence,gromov2007metric} that are relevant to our study are stated in the unrooted case (the measured metric spaces are not rooted). The adjunction of a root only induces straightforward modifications and we omit the details.

\subsection{Fractal properties of metric spaces, fractal dimensions}

We introduce notations and recall some definitions concerning fractal dimensions in a general setting, for a generic metric space $(X,d)$.

Let $\delta>0$, for a relatively compact subset $B \subset X$, let $N_B(\delta)$ be the smallest cardinality of a set of balls of radius $\delta$ covering $B$. The Minkowski dimension quantifies the power-law behaviour of $N_B(\delta)$ as $\delta$ tends to zero. More precisely, we differentiate the \emph{lower Minkowski dimension} $\dimmi(B)$ and the \emph{upper Minkowski dimension} $\dimms(B)$:
\begin{equation}\label{defdimm}
\dimmi(B):=\liminf\limits_{\delta \rightarrow 0} \frac{\log N_B(\delta)}{-\log \delta} , \quad \dimms(B):=\limsup\limits_{\delta \rightarrow 0} \frac{\log N_B(\delta)}{-\log \delta}.
\end{equation}

If both values coincide, we will denote the common value $\dimm (B)$ and call it the \emph{Minkowski dimension} of $B$.

Another fractal dimension can be defined as follows. We define first the family $\left(H^s\right)_{s >0}$ of \emph{$s$-dimensional Hausdorff measures} given for any $A \subset X$ by:
\begin{equation*}
H^s\left(A\right):= \lim\limits_{\delta \rightarrow 0} \left\{ \sum\limits_{i \geq 1} \vert U_i \vert^s : A \subset \bigcup_{i \geq 1} U_i \mbox{ and } \vert U_i \vert \leq \delta \mbox{ for all } i \geq 1 \right\},
\end{equation*}
where $\vert U \vert:= \sup \{ d(x,y) : x,y \in U\}$ denotes the diameter of $U$ for any $U \subset S$. The \emph{Hausdorff dimension} of $A$ is then defined by:
\begin{equation*}
\dimh \left( A \right)= \inf \left\{s \geq 0 : H^s(A)=0\right\}.
\end{equation*}

The following property will be useful in order to get a lower bound on the \emph{Hausdorff dimension} of a set. We refer to \cite[Proposition 4.9]{falconer1990fractal} or \cite[Theorems 6.9 and 6.11]{mattila1995geometry} for formulations in $\mathbb R^d$ which both extend to separable metric spaces.
Let $B_r(x):=\left\{y \in X : d(x,y)<r\right\}$ denote the open ball of radius $r$ centered in $x$. Then for a measurable set $A \subset X$, a finite measure $\nu$ on $X$ with $\nu(A)>0$ and $c>0$, we have
\begin{equation}\label{frostman}
\limsup\limits_{r \rightarrow 0} \frac{\nu \left(B_r(x)\right)}{r^s} \leq c \mbox{ for all } x \in A \quad \Rightarrow \quad \dimh(A)\geq s.
\end{equation}

Furthermore, for a bounded non-empty set $B$, we have
\begin{equation}\label{eq:reldimf}
\dimh(B) \leq \dimmi(B) \leq \dimms(B).
\end{equation}

\subsection{Dirichlet and Poisson-Dirichlet distributions}
\label{dirpoidir}
We recall definitions and fix notations for several probability distributions that will appear further on.

For $a,b>0$, we write $\lbeta(a,b)$ for the beta distribution with parameters $a,b$, whose dentisty with respect to the Lebesgue measure on $(0,1)$ is given by
\begin{equation*}
\frac{\Gamma(a+b)}{\Gamma(a) \Gamma(b)} x^{a-1}(1-x)^{b-1}.
\end{equation*}

A multivariate counterpart of the beta distribution is the \emph{Dirichlet distribution}. For $\beta := (\beta_1,\ldots \beta_n)$ a family of positive real numbers, the \emph{Dirichlet distribution} with parameter $ \beta $, denoted $\dir( \beta )$ is a distribution on the $(n-1)$-dimensional simplex $\Sigma_n:= \big\{ \mathbf{x} = (x_i) \in (0,1)^n : \: \sum_{i =1}^n x_i =1 \big\}$. Its density with respect to the $(n-1)$-dimensional Lebesgue measure is given by
\begin{equation*}
\frac{\Gamma\left(\sum_{i=1}^n \beta_i\right)}{\prod_{i=1}^n \Gamma(\beta_i)} \prod_{i=1}^n x_i^{\beta_i-1}.
\end{equation*}

Let us now define the \emph{Poisson-Dirichlet} distribution with parameters $a,b$ satisfying $a \in (0,1)$ and $b>-a$, denoted $\poidir(a,b)$. We follow Pitman and Yor in \cite{pitmanyor1997two}. First, we consider a family $W_i$ of independent random variables with $W_i \sim \lbeta(1-a,b+ja)$ for each $i \geq 1$. We then construct the sequence of random variables $(P_i)$ where
\begin{equation*}
P_i:= W_i \prod_{j=1}^{i-1} \left(1-W_j\right)
\end{equation*}
and reorder this sequence in the decreasing order to obtain $(P_i^\downarrow)_{i \geq 1}$. This sequence is said to have $\poidir(a,b)$ distribution, which takes values in the set $\Sigma^\downarrow:= \big\{ (x_i)_{i \in \mathbb Z^+} : \: 1\geq x_1 \geq x_2 \geq \ldots \geq 0 \mbox{ and } \sum_i x_i =1 \big\}$. A size-biased pick $V$ from a sequence $(P_i^\downarrow)$ with $\poidir(a,b)$ distribution has $\lbeta(1-a,a+b)$ distribution.

For more details about \emph{Poisson-Dirichlet} distributions, we refer to \cite[Chapter 3]{pitman2006combinatorial} for links with an urn process and to Chapter 4 for a description in term of lengths of excursions of Lévy processes.


We also consider the degenerate case where $b=-a$: we set $\poidir(a,-a)$ to be a dirac distribution on the sequence $(1, 0 ,0 , \ldots)$.

Given a sequence $(P_i^\downarrow)_i$ with $\poidir(a,b)$ distribution with $b>-a$, the $a$-diversity is defined as the almost sure limit
\begin{equation}\label{diversity}
\Gamma(1-a) \lim_{n \rightarrow +\infty} n (P_n^\downarrow)^a,
\end{equation}
which exists almost surely. This random variable has density $\mittagl (a,b)$ which stands for the \emph{General Mittag-Leffler distribution}. We refer again to Pitman \cite{pitman2006combinatorial} Section 0.3 for a definition, and to Chapter 3 for links with the \emph{Poisson-Dirichlet} distributions. An important property of \emph{Mittag-Leffler} distribution is that they have finite moments of all orders.

\section{Setting}

We define the general model in Subsection \ref{sectiondescriptionmodele}, then we precise hypotheses concerning the recursion itself together with hypotheses concerning the geometry in Subsection \ref{sec:hyp}. Lastly we discuss other models in Subsection \ref{sec:related}.

\subsection{A recursive decomposition of metric spaces}

\label{sectiondescriptionmodele}
In this section, we introduce a general model to describe metric spaces by a recursive decomposition. We make rigorous the idea represented in Figures \ref{fig:intro} and \ref{exconcretphi}: we glue rescaled copies of a random metric space onto random blocks, following a tree structure.

Let $\Gamma$ be a rooted plane tree with $K$ vertices where $K \in \overline{\mathbb Z^+} = \mathbb Z^+ \cup \{+\infty\}$. We say that $\Gamma$ is the \emph{structural tree} of our construction.
Let $\mathbf{r}, \mathbf{s} \in \Sigma_\Gamma := \big\{ \mathbf{x} = (x_i) \in (0,1)^\Gamma : \: \sum_{i \in \Gamma} x_i =1 \big\}$, and let $\alpha \in (0,1)$.

For each vertex $k \in \Gamma$, let $(B_k,\ell_k)$ be a compact metric space with marked points denoted by $b_{k,i}$ for any $i \in \kappa_k \cup \{k\}$ : hence there are as many points as the degree of $k$ in the tree $\Gamma$, plus $1$ for the root (recall the notations from Section \ref{section2preliminaries}). These metric spaces will be referred as "blocks" or "added length", which will be motivated after the description of the recursion. A simple yet fundamental example is the case when a block is a circle homeomorphic to $\mathbb S^1$.

{\old
}

Given such parameters $\Gamma$, $\mathbf{r}$, $\mathbf{s}$, $\alpha$ and $(B_k,\ell_k)_k$, we can now describe our construction step by step. Let $\mathfrak{X}_i := \left(X_i,d_i,\mu_i,\rho_i \right)$, $i \in \Gamma$, be rooted measured metric spaces to which we refer as the \emph{input} spaces, we construct $\mathcal{X} := \left(X,d,\mu,\rho \right)$ as follows:
\begin{enumerate}
    \item  Let $\eta_i$ be independent random variables, where each $\eta_i$ is sampled on $X_i$ according to the probability measure $\mu_i$.
    \item \label{construction_ii} Let $X^\circ$ be the disjoint union: $X^\circ:= \left(\bigsqcup\limits_{i \in \Gamma} X_i\right) \bigsqcup \left(\bigsqcup\limits_{i \in \Gamma} B_i\right)$.
    \item Let $d^\circ$ be the maximal pseudometric on $X^\circ$ such that the restriction of $d^\circ$ on $X_i$ is no greater than $r_i^\alpha d_i$ for all $i \in \Gamma$ and no greater than $\ell_i$ when restricted on each $B_i$. In addition, we require that for each $i \in \Gamma$, $d^\circ(\rho_i,b_{i-,i})=0$ and $d^\circ(\eta_i,b_{i,i})=0$. See for instance \cite{buragoburagoivanov} Section 3.1.3 which ensures existence and uniqueness of $d^\circ$ defined by gluing.
    \item Let $\mu^\circ$ be the unique probability measure on $X^\circ$ compatible with $s_i\mu_i$ on $X_i$ for every $i \in \Gamma$.
    \item Let $X:=X^\circ / \sim$ where $x \sim y$ if and only if $d^\circ(x,y)=0$. Let $d$ be the metric induced by $d^\circ$ on $X$. Let $\mu$ be the push-forward of $\mu^\circ$ under the canonical projection $\varphi^\circ$ from $X^\circ$ onto $X$.
    \item Lastly let $\rho:=\varphi^\circ (\rho_1)$ be the root.
\end{enumerate}

\begin{figure}
\begin{center}
	\begin{tikzpicture}[scale=.7]
		\node[draw,circle,fill=black] (1) at (0,0) {};
		\node[draw,circle] (2) at (-1,1) {};
		\node[draw,circle] (3) at (-1.5,2) {};
		\node[draw,circle] (4) at (-.5,2) {};
		\node[draw,circle] (5) at (0,1) {};
		\node[draw,circle] (6) at (1,1) {};
		\node[draw,circle] (7) at (1,2) {};
		\node (8) at (1,0) {$\Gamma$};
		
		\draw (1)--(2);
		\draw (2)--(3);
		\draw (2)--(4);
		\draw (1)--(5);
		\draw (1)--(6);
		\draw (6)--(7);
	\end{tikzpicture}
	\includegraphics[scale=.25]{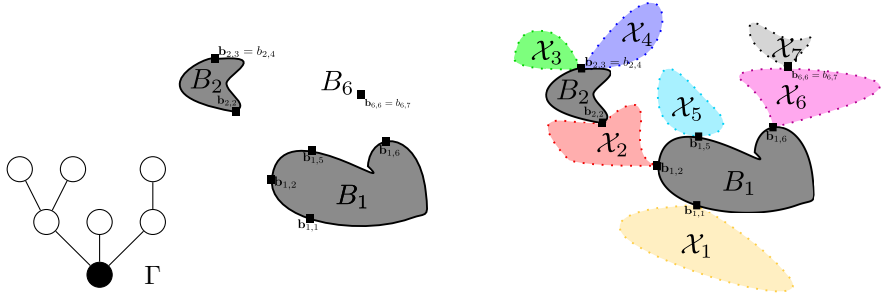}
\end{center}
	\caption{\label{exconcretphi} The construction of $\mathcal X$ with input spaces $(\mathfrak X_i)_{1\leq i\leq 7}$ with the structural tree $\Gamma$, the relevant blocks are described in the middle drawing with the marked points onto them (we write for instance $b_{2,3}=b_{2,4}$ to illustrate that the distance between the points is $0$ in the corresponding block). Parameters $\mathcal R, \mathcal S$ are not described here.}
\end{figure}

The $\gp$-equivalence class of the output $\mathcal X$ is random since the points $(\eta_i)$ were chosen randomly in the process. Moreover, it is a random variable (one may adapt Lemma 29 from \cite{broutin2016self}). From the point of view in Section \ref{matrixdist} with the distance matrix distribution, this random variable only depends on the $\gp$-equivalence classes of the input spaces $\left(\mathfrak X_i\right)_{i \in \Gamma}$ together with the distances between each couple of marked points in the metric spaces $(B_k,\ell_k)$. Indeed, the $\gp$-equivalence class only depends on distances between points picked under the measure of the space and the blocks $B_k$ for $k \in \Gamma$ are not weighted.

To be more explicit about the construction, we compute here the distance between the points $x \in \varphi^\circ(X_i)$ and $y \in \varphi^\circ(X_j)$ in the output space $\mathcal X$. Recall that regarding the $\gp$-isometry classes, we only need these distances and do not require to specify distances between points which are not weighted (e.g. the points inside the blocks).
Two situations have to be considered, depending on whether one of the vertices $i,j$ is an ancestor of the other one in $\Gamma$.
\paragraph{ If $i \notin \Gamma_j$ and $j \notin \Gamma_i$} The distance between $x$ and $y$ is given by:
\begin{multline}\label{eq:descrdist1}
d(x,y)= \sum_{k\in \pi_{i,j}} \ell_{k-}(b_{k-,k},b_{k-,k-}) +\ell_{q(i,j)} (b_{q(i,j) , i} , b_{q(i,j) , j })  \\+ r_i^\alpha d_i(\rho_i,x) + r_j^\alpha d_j(\rho_j,y)
+ \sum\limits_{k \in \pi_{i,j} \: k \neq i,j} r_k^\alpha d_k(\rho_k,\eta_k),
\end{multline}
where the  notation $\pi_{i,j}$ was defined in Section \ref{notapath}: the path between two vertices excluding their closest common ancestor $q(i,j)$ is the set of vertices to consider since the geodesic from $x$ to $y$ does not "cross" the space $\varphi^\circ (X_{q(i,j)})$ in this case.

\paragraph{If $j \in \Gamma_i$} The symmetric case $i \in \Gamma_j$ being the same, we essentially have to change $\rho_i$ into $\eta_i$ for the distance inside $\varphi^\circ (X_i)$ since the branching "goes up" and get:
\begin{multline}\label{eq:descrdist2}
d(x,y)=\sum_{k\in \pi_{i,j}} \ell_{k-}(b_{k-,k},b_{k-,k-})  + r_i^\alpha d_i(\eta_i,x) + r_j^\alpha d_j(\rho_j,y)\\+ \sum\limits_{k \in \pi_{i,j} \: k \neq i,j} r_k^\alpha d_k(\rho_k,\eta_k).
\end{multline}

The distance between any couple of points in the image of an input space by the canonical projection $\varphi^\circ$ can always be decomposed according to the former or the latter expression, or by $r_i^\alpha d_i(x,y)$ if they fall in the same $\varphi^\circ (X_i)$.

We consider the following parallel construction to the one described earlier in this section: keep the same parameters $\Gamma, \mathbf r, \mathbf s, \alpha, (B_k,\ell_k)_k$ as previously but take singletons for the input spaces $\mathfrak{X_i}$. The output will then be a rooted measured metric space which consists in the spaces $B_i$ glued together so that for each $i\in \Gamma$, $b_{i,i}$ is identified with $b_{i-,i}$: write $b_i$ for their common canonical projection. We denote this space by $(B, \ell)$ on which we still have marked points $(b_i)_{i \in \Gamma}$. It is rooted in $b_\emptyset$ and is endowed with a discrete probability measure that allocates mass $s_i$ to $b_i$ for all $i \in \Gamma$.

From Section \ref{matrixdist}, the $\gp$-equivalence class of $B$ is characterized by  the masses $\mathbf s$ (which already appear in the parameters of the recursion) together with the matrix $L$ with coefficients $L(i,j)=\ell(b_i,b_j)$ for all $(i,j) \in \Gamma^2$. For this sake, we introduce the notation $D_\Gamma$ for the corresponding space of infinite positive symmetric matrices so that the tree structure of $\Gamma$ is respected.

\begin{lrem} The description of this metric space $(B ,\ell)$ is crucial in our study. As precised before, the $\gp$-equivalence class of the resulting space $\mathcal X$ (with general input spaces $\mathfrak X_i$), when it comes to the dependence to the blocks, depends only on the $\gp$-equivalence class of $(B,\ell, \sum_i s_i \delta_{b_i},b_\emptyset)$  and not on the entire metric spaces $(B_i)$ for all $i\in \Gamma$.
\begin{itemize}
\item If $B_\emptyset$ is a sphere, but the points $b_{\emptyset,i}$ are all on the equator, then everything will work as if the initial space was a circle isometric to $\mathbb S^1$.
\item The $\gp$-equivalence class of $B$ is equivalently described by a completion of the subspace countaining the geodesics between any two points $b_i,b_j$.
\end{itemize}
\end{lrem}

With these new notations, both expressions for distances between points in the space $\mathcal X$ can be simplified. Equations \eqref{eq:descrdist1} and \eqref{eq:descrdist2} give respectively
\begin{equation*}
d(x,y)= L(i,j)+ r_i^\alpha d_i(\rho_i,x) + r_j^\alpha d_j(\rho_j,y)
+ \sum\limits_{k \in \pi_{i,j} \: k \neq i,j} r_k^\alpha d_k(\rho_k,\eta_k),
\end{equation*}
and
\begin{equation*}
d(x,y)=L(i,j)+ r_i^\alpha d_i(\eta_i,x) + r_j^\alpha d_j(\rho_j,y)\\+ \sum\limits_{k \in \pi_{i,j} \: k \neq i,j} r_k^\alpha d_k(\rho_k,\eta_k).
\end{equation*}

The application $\psi:(\kgp)^\Gamma \times \Sigma_\Gamma^2  \times D_\Gamma \rightarrow \mathcal M_1 (\kgp)$ taking as arguments the $\gp$-equivalence classes $\left(\mathfrak X_i\right)_{i \in \Gamma}$, the rescaling parameters $\mathbf r, \mathbf s$ and also the distance matrix $L$ described above is well-defined and measurable when endowing the space of probability measures on $\kgp$ , $\mathcal M_1(\kgp)$, with the Prokhorov distance (see Lemma \ref{lem:mes} in Appendix). For probability measures $\tau$ on $\Sigma_\Gamma^2\times\mathcal D_\Gamma$, and $\aleph$ on $\kgp$, we define the annealed measure:
\begin{equation*}
\Psi \left(\aleph, \tau \right) \left(A \right) := \mathbb E \left[ \psi \left(\left( \mathcal X_i \right)_i,  \mathcal R, \mathcal S , \mathcal L\right) \left(A\right) \right],
\end{equation*}
for all measurable set $A \subset \kgp$, where $(\mathcal R, \mathcal S,\mathcal L)$ has distribution $\tau$, and the $\mathcal X_i$ are iid random variables with distribution $\aleph$ also independent from $(\mathcal R, \mathcal S, \mathcal L)$.

Given the parameters $\Gamma, \alpha, \tau$, we want to study distributions $\aleph$ which are fixed points of $\Psi$, or said in other words such that:
\begin{equation}
\label{eqnppl}
    \aleph = \Psi \left( \aleph, \tau  \right),
\end{equation}
where we omit to specify the dependance in $\alpha$ and $\Gamma$ since the latter one is fixed, and $\alpha$ needs to verify specific conditions described later in Section \ref{sec:hyp}.

Let us observe that we do not assume that the metric spaces $(B_k,\ell_k)$ are trees or even tree-like: if it exists, a solution of Equation \eqref{eqnppl} may have cycles as for example in Figure \ref{fig:intro}, or the $\beta$-stable looptrees which appear later in Section \ref{subsubs:lt}. It is a significant difference with what has been done before because the metric spaces we study are not real trees in general, hence we cannot use the correspondence between real trees and continuous excursions which is the general framework in the litterature. See for instance \cite{broutin2016self,le2005random}.

\subsection{Hypotheses}
\label{sec:hyp}

We make three main hypotheses to corner our model. Hypothesis \eqref{hypLnonnul} ensures that some length is added with positive probability by the recursion, and Hypothesis \eqref{condsufeqnY} with $p=1$ that a relation concerning distances is actually a contraction. The last main assumption, corresponding to both Hypothesis \eqref{hypdimms} and Hypothesis \eqref{hypdimhblock}, makes sure that the random blocks on which are glued the rescaled iid copies of the input space have almost surely constant fractal dimensions.

Other assumptions, less core and more technical are also regrouped in this section.

\paragraph{Added length}We make the following assumption on the distribution of the lengths $\mathcal L$:
\begin{EqHypothese} \label{hypLnonnul}
\bbP \left( \mathcal L(\emptyset,i)>0 \mbox{ for some } i\in \Gamma \right)>0.
\end{EqHypothese}
This assumption means that $\mathcal L$ actually adds a distance at some point. It is equivalent to the fact that at least one distance $\mathcal L (i,j)$ is non zero with positive probability.

To deal with random $\gp$-equivalence classes of rooted measured metric spaces $\mathcal X = (\mathfrak X , d , \mu , \rho )$,
the height of random points $\zeta$ sampled according to the measure $\mu$, that we denote by $Y := d(\rho,\zeta)$ will play a key role. Indeed, it appears in the distance matrix defined in Section \ref{matrixdist}. If $\mathcal X$ is a solution of the fixed-point equation \eqref{eqnppl} in distribution, then the random variable $Y$ must satisfy:

\begin{equation}\label{eqndptunif}
\begin{split}
    Y &\loi \sum_{i \in \Gamma} \mathbf 1_{J =i} \left( \sum_{k \in \pi(\emptyset,i) \cup \{\emptyset\}} \mathcal R_k^\alpha Y^{(k)} + \mathcal L (\emptyset,i) \right) \\ 
    &\loi \sum\limits_{i \in \Gamma} \mathbf 1_{J \in \Gamma_i} \mathcal R_i^\alpha Y^{(i)} + \mathcal L\left( \emptyset,J\right),
\end{split}
\end{equation}
where $J$ is a random variable such that $\mathbb P (J =i \vert \mathcal R, \mathcal S,\mathcal L) = \mathcal S_i$ for all $i\in \Gamma$, the random variables $Y^{(i)}$ are iid copies of $Y$, and $(\mathcal R,\mathcal S,\mathcal L), Y^{(1)},  Y^{(2)}, \ldots$ are independent. We remind from Section \ref{notasubtree} that $\Gamma_i$ is defined as the subtree of $\Gamma$ rooted at $i$, and $\pi(i,j)$ is the set of vertices in the only path between $i$ and $j$ in $\Gamma$, to which we remove $q(i,j)$ the most recent common ancestor of $i,j$. Equation \eqref{eqndptunif} is studied in \cite{jelenkovic2012implicit,alsmeyer2012fixed} under the name of \emph{inhomogeneous smoothing transform}, as opposed to the notion introduced by Durrett and Liggett \cite{DurrettLiggett1983} in the \emph{homogeneous} case corresponding to the equation without righ-hand side (it would be $\mathcal L(\emptyset,J)=0$ almost surely and $K$ finite). We will follow their ideas in Section \ref{sectionminimalsolu} and our setting will require additional hypothesis concerning the distribution $\tau$.

Alsmeyer and Meiners give a necessary condition for an equation such as Equation \eqref{eqndptunif} to have non-trivial fixed-points in \cite[Theorem 6.1]{alsmeyer2012fixed}, which means in this context that the fixed-points are not almost surely infinite:
\begin{equation}
\label{condnecsumR}
\bbE \left[ \sum\limits_{i \in \Gamma} \mathbf 1_{J \in \Gamma_i} \mathcal R_i^\alpha \right] \leq 1.
\end{equation}

\paragraph{Contraction} We make the assumption that a stronger condition is satisfied. We state it for a parameter $p \geq 1$:
\begin{EqHypothese}
\bbE \left[ \left(\sum\limits_{i \in \Gamma} \mathbf 1_{J \in \Gamma_i} \mathcal R_i^\alpha \right)^p \right] < 1 \mbox{ and } \bbE \left[\mathcal L \left(\emptyset,J\right)^p\right]<+\infty. \tag{H$2_p$} \label{condsufeqnY}
\end{EqHypothese}
For $p=1$ this is sufficient in order to have non-trivial fixed-points to Equation \eqref{eqndptunif} from \cite[Lemma 4.4]{jelenkovic2012implicit}.

In order to obtain bounds on moments of the height of a fixed-point, we will sometimes assume that for for a $p \geq \frac{1}{\alpha}$
\begin{EqHypothese}
\bbE\left[\sup \left\{ \left(\sum\limits_{j \neq i : i\in \Gamma_j} \mathcal R_j^\alpha \right)^p: i \in \Gamma\right\}\right]<+\infty. \tag{H$3_p$} \label{hyphaut}
\end{EqHypothese}

\paragraph{Geometry} Let $\mathcal X$ be a random variable whose distribution is a solution of Equation \eqref{eqnppl} with fixed parameters. We write $\psi(\mathcal X)$ for a random variable defined as the image by $\psi$ of iid copies of $\mathcal X$ and $\left(\mathcal R, \mathcal S , \mathcal L\right)$ independent with distribution $\tau$. $\psi(\mathcal X)$ is then distributed as a copy of $\mathcal X$, and we are able to decompose it into the union of the random metric spaces $\mathcal X_i$, rescaled, together with blocks  $(B_k,\ell_k)$ intertwined in between (recall the construction in Section \ref{sectiondescriptionmodele}).

To be more precise, the rescaled random copies of $\mathcal X$ are glued onto and in between points of $\mathcal B$ which stands for the random version of the block $(B,\ell)$ introduced in Section \ref{sectiondescriptionmodele}. From this remark, we can see that the geometry of $\mathcal X$ depends on the geometry of $\mathcal B$.

\begin{lrem}
We will suppose that $\bbE \left[ \sup_i \mathcal L(\emptyset,i)^p \right]<\infty$ for a $p\geq \frac{1}{\alpha}$. This implies that an element of the $\gp$-isometry class of $\mathcal B$ is compact a.s.
\end{lrem}

The first assumption we make on the geometry of $\mathcal B$ is the following:
\begin{EqHypothese}\label{hypdimms}
\dimms(\mathcal B)= \dup \geq 0 \text{ a.s.}
\end{EqHypothese}
Note that this constant can be equal to $0$, for instance when the tree $\Gamma$ is finite.

Lastly we make a similar assumption concerning the Hausdorff dimension:
\begin{EqHypothese}\label{hypdimhblock}
\dimh (\mathcal B) = \dlow \geq 0 \text{ a.s.}
\end{EqHypothese}
We recall that $\dlow \leq \dup$.

\subsection{Related models}
\label{sec:related}
Several models in the litterature can be described with the general construction of Section \ref{sectiondescriptionmodele}. The specificities of our most basic setup are precisely cornered by the hypotheses stated in Section \ref{sec:hyp}, especially Hypotheses \eqref{hypLnonnul} and \eqref{condsufeqnY} with $p=1$ as we will make it clear in the results section.

The first description of a random metric space as the unique solution of a distributional fixed-point equation can be traced back to the article \cite{albenque2015brownian} by Albenque and Goldschmidt. In this article, the authors characterize the Brownian CRT by a decomposition in three pieces following the result of Aldous in \cite{aldous1994recursive}. They show that this fixed point is attractive by iteratively applying the decomposition. Results are stated for \emph{Gromov-Hausdorff} equivalence classes of compact metric spaces, and the convergence holds for the $\gp$ topology.

The article \cite{broutin2016self} by Broutin and Sulzbach somewhat enlarges their scope.
They are able to prove existence, uniqueness and attractiveness of a wider collection of random compact rooted measured metric spaces, in the sense of \emph{Gromov-Hausdorff-Prokhorov}, and are also able to characterize Minkowski and Hausdorff dimensions of the fixed-points under mild conditions.
With the notations from Section \ref{sectiondescriptionmodele}, their setup differs from ours since it requires a finite structural tree $\Gamma$ and blocks $(B_k,\ell_k)$ all reduced to singletons:
they identify points onto several of the iid copies instead of gluing them onto random blocks.
It entails that Hypothesis \eqref{hypLnonnul} is not satisfied in their study, and the computation of the height of a random point leads to a \emph{smoothing transform}:
\begin{equation}
\bbE \left[ \sum\limits_{i \in \Gamma} \mathbf 1_{J \in \Gamma_i} \mathcal R_i^\alpha \right] = 1,
\end{equation}
with the same notations as for Hypothesis \eqref{condsufeqnY}. This fixes the value of $\alpha$ and is known to admit a unique solution if the mean is given (see \cite{durrett1983fixed}).
Their setup covers applications both to the formerly studied Brownian CRT, and to the dual trees of random recursive triangulations of the disk introduced by Curien and Le Gall in \cite{curien2011random}.

Another work to mention is the article from Chee, Rembart and Winkel \cite{chee2018recursive}, concerning the $\beta$-stable trees. It is inspired by Marchal's random growth algorithm introduced in \cite{Marchal2008note}: the stable trees are the unique fixed point of a recursive equation that consists in gluing together infinitely many rescaled copies of a random compact metric space at a single branchpoint.
The authors obtain existence, uniqueness of the fixed point together with attractiveness for a larger class of recursive distributional equations. 
This model falls again out of our scope since Hypothesis \eqref{hypLnonnul} is not satisfied. Note also that they work with Gromov-Hausdorff equivalence classes on real trees.

It seems that there is no easy way to unify our setup with the three models cited above. Our study is complementary to theirs.

In the anterior paper \cite{MR3846837}, two of the same authors give a construction of continuum random trees via a recursive method which also states existence and uniqueness of solutions ot an equivalent of Equation \eqref{eqnppl}.
In their setup, a countably infinite number of iid copies are glued onto a random \emph{string of beads}, defined as a random interval endowed with a random discrete probability measure.
If we omit differences concerning the ambient metric spaces -- they work with \emph{Gromov-Hausdorff-Prohorov} equivalence classes restricted to real trees where the measure has full support and such that the height has $p$-moments -- our results generalize their construction (except for the \emph{generalized strings} which give a mass on the branch).
Indeed it corresponds to the special case of the star tree $\Gamma$ where infinitely many vertices are children of the root, and $(B_\emptyset, \ell_\emptyset)$, the only block that appears in the recursion, is a random \emph{string of beads}. They are able to compute the almost sure Hausdorff dimension of the fixed point, and apply their results to the $\beta$-stable trees.

Another work to mention is the article by Sénizergues \cite{senizergues2019random}. Metric spaces are constructed as an iterative gluing, generalizing stick-breaking construction with random "blocks" more general than the usual segments added at each step. It entails that metric spaces which are almost surely not trees arise, at each step since the blocks are not forced to be trees, and also at the limit. The main difference with our work is the iterative nature of the construction, where our construction is recursive.

\section{Main results}
\label{sectionresults}
\subsection{Existence and uniqueness of solutions to Equation \eqref{eqnppl}}

Recall the settings of Section \ref{sectiondescriptionmodele} in which we need to introduce a structural tree $\Gamma$, a distribution $\tau$ on $\Sigma_K^2 \times D_\Gamma$ and a parameter $\alpha \in (0,1)$.
We define $\munp{p}$ as the set of probability measures $\delta$ on $\kgp$ such that $\int_{\kgp} \int_{\mathcal X} d(\rho,\zeta)^p \mathrm d \mu(\zeta) \mathrm d \delta (\mathcal X,d,\mu,\rho) < \infty$. This integral can also be interpreted as $\bbE \left[d(\rho,\zeta)^p\right]$ for a random variable $\mathcal X=(X,d,\rho,\mu)$ with distribution $\delta$ and $\zeta$ sampled according to $\mu$, where the mean is taken over both $\delta$ and $\mu$.

\begin{lthe} \label{thmppl}
Assume that the parameters of Equation \eqref{eqnppl}, namely $\Gamma, \tau$ and $\alpha$ satisfy both Hypotheses \eqref{hypLnonnul} and \eqref{condsufeqnY} with $p=1$, then Equation \eqref{eqnppl} admits a unique solution $\eta$ on $\munp{1}$.
\end{lthe}

The setup considered for the fixed-point equation \eqref{eqnppl} does not differentiate spaces with massless portions. Thus, if we recall the remark at the end of Section \ref{matrixdist} the natural way to describe it is by the use of $\gp$-isometry classes. Results in a stronger sense, for instance with the Gromov-Hausdorff-Prohorov metric for which we refer to \cite{abraham2013note,broutin2016self}, would require at least to restrict ourselves to equivalence classes of compact measured rooted metric spaces with respect to this distance, such that the measure has full support: it could then extend our results in light of the discussion in \cite{broutin2016self} Section 2.1 about the $\gp$ topology.
Note that as remarked in \cite{broutin2016self}, this would not forbid the fact that there are other fixed points with massless paths.

The existence and uniqueness of a solution to Equation \eqref{eqndptunif} introduced in Section \ref{sec:hyp} is the keystone of the proof of Theorem \ref{thmppl}.

In order to obtain such a result, statement are made in the space $\munp{1}$, which means that for a random variable $\mathcal X =(X,d,\mu,\rho)$ with distribution $\eta$, the distance between the root and a random point sampled according to $\mu$ is integrable (when we average on both $\mu$ and $\eta$). Hypotheses \eqref{hypLnonnul} and \eqref{condsufeqnY} with $p=1$ give a sufficient condition for the existence of a solution of Equation \eqref{eqndptunif}, in addition to the $L^1$ property of the solution which will be a technical issue in the proof.

More can be said on the distribution $\eta$ under further hypothesis:
\begin{lthe}\label{Corpfmoment}
With $\eta \in \munp{1}$ the fixed-point solution from Theorem \ref{thmppl}:
\begin{enumerate}
\item[(i)] Under hypothesis \eqref{condsufeqnY} for a $p\geq 1$, then $\eta \in \munp{p}$.
\item[(ii)]\label{cor:ii} If additionally Hypothesis \eqref{hyphaut} holds and $\bbE \left[ \sup \left\{ \mathcal L(\emptyset,i)^p : i \in \Gamma \right\} \right]< + \infty$ for a $p\geq \frac{1}{\alpha}$, then for a random variable $\mathcal X$ with distribution $\eta$, we have $\bbE \left[ h(\mathcal X)^p \right]<+\infty$.
\end{enumerate}
\end{lthe}

The first result stated by this theorem is a straightforward application of Lemma 4.4 in \cite{jelenkovic2012implicit}. For the second part, Hypothesis \eqref{hyphaut} is minimal \footnote{The particular case where the tree $\Gamma$ is $\mathbb N$ in which $i$ is the only child of $i-1$ for any $i$ gives an example.} but in practice we use a simpler condition which will be verified in the main examples that we present:
\begin{lpro}\label{propfiniteH}
Hypothesis \eqref{hyphaut} holds for any $p$ if the structural tree $\Gamma$ has finite height. 
\end{lpro}
The proof is straightforward, since each term over which we take the supremum is almost surely bounded by a power of the height of $\Gamma$.

The next result deals with convergence of iterated distributions towards the fixed-point of Theorem \ref{thmppl}. Let $\phi : \mathcal M_1 (\kgp) \rightarrow \mathcal M_1 (\kgp)$ be defined as $\phi (\aleph) := \Psi (\aleph, \tau)$ for any $\aleph \in \mathcal M_1 (\kgp)$ where $\Psi$ is the application of Equation \eqref{eqnppl} and $\tau$ is the given probability distribution on $(\Sigma_\Gamma)^2 \times D_\Gamma$. Let also $\phi^n:= \phi \circ \phi \circ \ldots \circ \phi$ the $n$-th iterated of $\phi$ for $n \geq 1$.

\begin{lthe} \label{thmcviteres}
Let $\phi$ be as stated above, we suppose that the conditions \eqref{hypLnonnul} and \eqref{condsufeqnY} with $p=1$ are fulfilled. Let $\delta \in \munp{1}$. Then the sequence $(\phi^n(\delta))_n$ converges weakly to $\eta$.
\end{lthe}

\subsection{Bounds on fractal dimensions of the fixed-points}
\label{sec:resultdimf}

\begin{lthe} \label{thmborndimms}
Let $\mathcal X$ be a random variable with distribution $\eta$ from Theorem \ref{thmppl}, such that the conclusion of Theorem \ref{cor:ii} $(ii)$ is satisfied. Under the additional hypotheses \eqref{hyphaut} with $p\geq \frac{1}{\alpha}$ and \eqref{hypdimms}, we have
\begin{equation*}
\dimms (\mathcal X) \leq \max \left\{ \frac{1}{\alpha} , \dup \right\}.
\end{equation*}
\end{lthe}

Now let $\mathcal R_*$ be a random variable describing a size-biased pick among the $\mathcal R_i$-s: let $I$ be a random variable on $\Gamma$ such that $\mathbb P \left(I=k \vert \mathcal R, \mathcal S ,\mathcal L \right)=\mathcal R_k$ for all $k \in \Gamma$ and set $\mathcal R_*:=\mathcal R_{I}=\sum_{i \in \Gamma} \mathbf 1_{I=i} \mathcal R_i$.

\begin{lthe} \label{thmborndimh}
Let $\mathcal X$ be a random variable with distribution $\eta$ from Theorem \ref{thmppl}.  Under the additional hypotheses that $\mathbb E \left[\mathcal R_*^{-\delta}\right]<+\infty$ for a $\delta>0$, and under Hypothesis \eqref{hypdimhblock} we have the almost sure bound:
\begin{equation*}
\dimh (\mathcal X ) \geq \max \left\{ \dlow , \frac{1}{\alpha} \right\}.
\end{equation*}
\end{lthe}

\subsection{Application to classical examples in the litterature}

We recover two important classes of random metric spaces with our construction: the $\beta$-stable trees and the $\beta$-stable looptrees. For the distributions involved in both constructions, we refer to Section \ref{dirpoidir}.

\subsubsection{The $\beta$-stable looptrees of Curien and Kortchemski}
\label{subsubs:lt}

For general definitions about the $\beta$-stable looptrees with $\beta \in (1,2)$, we refer to the original paper \cite{curien2014random} which we summarize at the beginning of Section \ref{sectionlooptrees} before proving the following theorems. To give insight and motivation, they are constructed from a normalized excursion of a $\beta$-stable Lévy process with positive jumps, on which each jump of the process will code a loop in the looptree. Even if they are constructed as a random variable taking values in Gromov-Hausdorff isometry classes of compact metric spaces, it is possible to consider a $\kgp$ version of the $\beta$-stable looptrees as it is done later in Section \ref{sectionlooptrees}.

Looptrees arise naturally in the study of random planar maps and statistical mechanics  models on these maps, originally with \cite{LeGall2011Scaling} anterior to \cite{curien2014random}, in which Le Gall and Miermont implicitely encounter them.
Together with the original paper, Curien and Kortchemski prove in \cite{Curien2015Percolation} that a $\frac{3}{2}$-stable looptree appears as the scaling limit of the cluster boundaries in a critical percolation on the UIPT.
We refer more generally to the following non-exhaustive list \cite{Richier2018Limits,Kortchemski2020Boundary} on random maps.

Let $\Gamma$ be the plane tree drawn in Figure \ref{structuraltreelooptrees}: every vertex of index $i>0$ is a child of the root vertex $\emptyset$ which we will denote by $0$ so that $\Gamma$ is indexed by $\mathbb N$.

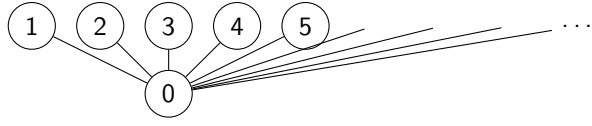
\begin{figure} 
\begin{center}
\begin{tikzpicture}[scale=.9]
\node[draw,circle] (0) at (2,0) {0};
\node[draw,circle] (1) at (0,1) {1};
\node[draw,circle] (2) at (1,1) {2};
\node[draw,circle] (3) at (2,1) {3};
\node[draw,circle] (4) at (3,1) {4};
\node[draw,circle] (5) at (4,1) {5};
\node (6) at (5,1) {};
\node (7) at (6,1) {};
\node (8) at (7,1) {};
\node (9) at (8,1) {\ldots};

\draw (0) -- (1);
\draw (0) -- (2);
\draw (0) -- (3);
\draw (0) -- (4);
\draw (0)--(5);
\draw (0)--(6);
\draw (0)--(7);
\draw (0)--(8);
\draw (0)--(9);
\end{tikzpicture}
\end{center}
\caption{The structural tree for the Looptrees\label{structuraltreelooptrees}}
\end{figure}

Let us define:
\begin{equation}\label{parameterslooptree}
\begin{split}
\left( X_0,X_1,X_2,X_3 \right) &\sim \dir \left(1-\frac{1}{\beta} , 1 - \frac{1}{\beta} , 1-\frac{1}{\beta} , \frac{2}{\beta}-1 \right), \\
\left(P_i^\downarrow \right)_{i \geq 0} &\sim \poidir \left( \frac{1}{\beta}, \frac{2}{\beta}-1  \right),
\end{split}
\end{equation}
chosen independently from each other. Now let us define for all $i \in \mathbb N$
$$
\mathcal R_i = \mathcal S_i := 
\left\{ 
\begin{array}{ll}
X_i & \mbox{if } i \leq 2 \\
X_3 P_{i-3}^\downarrow & \mbox{else}.
\end{array}
\right.
$$
We also define a random variable $W$ which is the \emph{$1/\beta$-diversity} of the decreasing sequence $(P_i^\downarrow)_i$:
\begin{equation*}
W:= \Gamma \left(1-\frac{1}{\beta} \right) \lim\limits_{i \rightarrow \infty} i (P_i^\downarrow)^{\frac{1}{\beta}}.
\end{equation*}

Let us consider
\begin{equation*}
\mathcal L(i,j) := X_3^{\frac{1}{\beta}} W \min \left\{ \vert U_i-U_j \vert , 1- \vert U_i-U_j \vert \right\} \mbox{ for every } (i,j) \in \mathbb N^2,
\end{equation*}
where $(U_i)_{i \in \mathbb N }$ is a sequence of iid random variables with uniform distribution on $[0,1]$, also independent from the formerly introduced random variables.

The last parameter we need is the scaling $\alpha \in (0,1)$ which we take equal to $\frac{1}{\beta}$.

\begin{lthe}
\label{thmlooptreev2}
The random $\beta$-stable looptree $\left( \mathscr L , d,\mu,\rho \right)$ satisfies the recursive distributional equation
\begin{equation} \label{looptreeqdistfpe}
\mathscr L \loi \psi \left( (\mathscr L_i)_{i \in \mathbb N}, (\mathcal R_i)_i , (\mathcal S_i)_i , \mathcal L \right) ,
\end{equation}
where $(\mathscr L_i)_i$ is a sequence of independent copies of $\mathscr L$, and independent of the random variables $(\mathcal R , \mathcal S , \mathcal L)$, and $\psi$ is the application from Section \ref{sectiondescriptionmodele} with the structural tree from Figure \ref{structuraltreelooptrees} and $\alpha=\frac{1}{\beta}$.
\end{lthe}

From there we can apply Theorem \ref{thmppl} to get a new characterization of the $\beta$-stable looptree:
\begin{lcor}\label{cor:ltpointfixe}
We can define $\mathscr L_\beta$ as the unique fixpoint in distribution in $\kgp$ to Equation \eqref{looptreeqdistfpe}.
\end{lcor}

We can also understand this theorem "backwards": sample two points on $\mathscr L_\beta$ independently, according to the measure $\mu$, and consider the loop at the intersection of the $3$ geodesics between the two points and the root $\rho$. This loop has length $X_3^{1 / \beta} W$ which is the local time of the coding stable process. Now if we remove it from the looptree, we obtain countably many connected components which were glued uniformly around the loop. The ones containing respectively $\rho$ and the randomly sampled points are rescaled copies of $\mathscr L_\beta$ with respective size $X_0,X_1,X_2$ from Equation \eqref{parameterslooptree}, and all the other ones are also independent scaled copies of $\mathscr L_\beta$ with size $\mathcal R_i$ for $i>2$.

One may recognize the same scaling paramaters $(\mathcal R_i)_i$ as in \cite{chee2018recursive} on the $\beta$-stable tree, already mentioned in Section \ref{sec:related}. Indeed these parameters are directly computed from the $\beta$-stable process coding both the $\beta$-stable tree and looptree.
\begin{lcor}\label{cordimLT}
Almost surely, for any $\beta \in (1,2)$, $$\dimms ( \mathscr L_\beta)=\dimh (\mathscr L_\beta)=\beta.$$
\end{lcor}

\subsubsection{The $\beta$-stable tree}

The $\beta$-stable tree for $\beta \in (1,2]$ studied by Duquesne and Le Gall in  \cite{duquesne2002random,duquesne2005probabilistic} are obtained as a special case of the Levy trees of Le Gall and Le Jan \cite{leGallleJan1998}, and generalize the Brownian CRT. Note that our results do not apply to the more general case of Lévy trees, which are not \emph{scale invariant} in general.

As for the stable looptrees, we consider $\mathcal T_\beta$ as a random variable in $\kgp$, and describe its distribution as a fixed-point of an equation like Equation \eqref{eqnppl}. We first describe the situation for $\beta \in (1,2)$, then we will add the case $\beta=2$ which corresponds to the Brownian CRT as a degenerate case.

We recall the results of Haas, Pitman and Winkel about the spinal mass partitions of a $\beta$-stable random tree from \cite[Corollary 10]{haas2009spinal}. To make an easier link between the notations, we take as our structural tree to define the application $\psi$ the same one as for the \emph{$\beta$-stable looptrees} drawn in Figure \ref{structuraltreelooptrees}, but we index the vertices by $\left(\mathbb Z^+ \right)^2$ so that $(1,1)$ is the root, and every $(i,j) \neq (1,1)$ is a child of the root.

Let $(P_i^\downarrow)_{i \geq 1}$ a sequence such that
\begin{equation*}
(P_i^\downarrow)_{i \geq 1} \sim \poidir \left(1-\frac{1}{\beta} , 1- \frac{1}{\beta} \right).
\end{equation*}
We note $W$ its $(1-1/\beta)$-diversity:
\begin{equation*}
W:= \Gamma\left(\frac{1}{\beta}\right) \lim\limits_{i \rightarrow +\infty} i(P_i^\downarrow)^{1-\frac{1}{\beta}},
\end{equation*}
where the limit exists almost surely.

Let $(U_i)_{i \geq 1}$ be iid random variables with uniform distribution on $[0,1]$,  independent of $(P_i^\downarrow)$.

For each $i \in \mathbb N$, we consider a random sequence
\begin{equation*}
(Q_j^{(i),\downarrow})_{j \geq 1} \sim \poidir \left(\frac{1}{\beta} , \frac{1}{\beta}-1 \right),
\end{equation*}
independent of each other and of the formerly introduced random variables.

We can now define our parameters on the structural tree indexed by $\left(\mathbb Z^+ \right)^2$:
\begin{equation*}
\begin{split}
\mathcal R_{i,j}&=\mathcal S_{i,j}:= P_i^\downarrow Q_j^{(i),\downarrow} , \mbox{ for all } (i,j) \in \left(\mathbb Z^+ \right)^2 \\
\mathcal L\left(\left(i,j\right) , \left(k,l \right) \right)& := \vert U_i-U_k \vert W \mbox{ for all } (i,j) \neq (k,l) \\
\alpha &:= 1-\frac{1}{\beta}.
\end{split}
\end{equation*}

Some comments are in order at this point. The block on which we glue the rescaled trees is a random interval with length $ W$, endowed with a discrete probability measure giving mass $P^\downarrow_i$ to the point at distant $W U_i$ from the origin. The second family of Poisson-Dirichlet distributions $Q^{(i),\downarrow}$ defined on each gluepoint of index $i$, describe the relative mass of the countably infinite number of trees glued at each node.

Lastly, let $(I,J)$ be a couple  of random variables in $\mathbb Z^+$ taken as a size biased pick from the $\mathcal R_{i,j}$s: for all $i,j$
\begin{equation*}
\mathbb P \left( (I,J)=(i,j) \vert \mathcal R \right)= \mathcal R_{i,j}.
\end{equation*}
We then swap the picked couple $(I,J)$ with $(1,1)$ so that the space at the root of the structural tree is rescaled with a sized-biased parameter from all the $\mathcal R_{i,j}$s.

This last step is more of a technicality to apply our setup as with the spinal mass decomposition from \cite{haas2009spinal}, the root is at the end of the random interval. In fact, it only consists in resampling a new root at random by using the invariance by rerooting of the $\beta$-stable tree (see \cite{duquesne2009rerooting}) in order to have the root inside one of the rescaled trees.

\begin{lthe}\cite[Corollary 10]{haas2009spinal} \label{thmarbrestable}
The random $\beta$-stable tree $\mathcal T_\beta$ satisfies the recursive distributional equation 
\begin{equation}\label{arbrebetastable}
\mathcal T_\beta \loi \psi \left( (\mathcal T_\beta ^{(i,j)})_{i,j \in \mathbb Z^+}, (\mathcal R_{i,j})_{i,j} , (\mathcal S_{i,j})_{i,j} , \mathcal L \right) ,
\end{equation}
where $(\mathcal T_\beta ^{(i,j)})_{i,j \in \mathbb Z^+}$ is a sequence of independent copies of $\mathcal T_\beta$, and independent of the random variables $(\mathcal R , \mathcal S , \mathcal L)$, and $\psi$ is constructed with the structural tree from Figure \ref{structuraltreelooptrees} indexed by $(\mathbb Z^+)^2$ and with $\alpha=1-\frac{1}{\beta}$.
\end{lthe}

Again we can apply Theorem \ref{thmppl} to get a new characterization of the $\beta$-stable tree:
\begin{lcor}\label{corstablet}
We can define $\mathcal T_\beta$ as the unique fixpoint in distribution in $\kgp$ to Equation \eqref{arbrebetastable}.
\end{lcor}

The case $\beta=2$ appears a simpler case of this fixed-point setup. Indeed, we can define all the parameters as before, with the only distinction of the family of Poisson-Dirichlet distributions $Q^{(i),\downarrow}$ all fit the degenerate case given in Section \ref{dirpoidir}, giving mass $1$ to the first index, and $0$ to all the subsequent. Hence only one branch is attached at each node of the interval which reminds the fact that the Brownian CRT is almost surely binary.
\begin{lthe}
The results from Theorem \ref{thmarbrestable} and Corollary \ref{corstablet} also hold in the case $\beta=2$.
\end{lthe}

We recover the results of \cite{duquesne2005probabilistic} concerning fractal dimensions:
\begin{lcor}\label{cordimT}
Almost surely, for any $\beta \in (1,2]$, $$\dimms (  \mathcal T_\beta)=\dimh (\mathcal T_\beta)=\frac{\beta}{\beta-1}.$$
\end{lcor}

\subsection{Coupling - Overview of the main techniques}
\label{sectioncoupling}

The main idea of the proofs is to expand the fixed-point Equation \eqref{eqnppl}, and to iterate the application with coupled randomness.

Let $\Theta:= \bigcup_{n\geq 0} \Gamma^n$, a complete infinite tree. On this tree, for each $n$, define $\Theta_n:= \Gamma^n \subset \Theta$ the subspace of vertices at level $n$ in $\Theta$.

With that in mind, we now introduce:
\begin{equation*}
\left\{ \left( \mathcal R^{(\theta)},\mathcal S^{(\theta)}, \mathcal L^{(\theta)} \right) , \theta \in \Theta \right\},
\end{equation*}
a set of independent random variables where each $\left( \mathcal R^{(\theta)},\mathcal S^{(\theta)}, \mathcal L^{(\theta)} \right)$ is an independent copy of $\left( \mathcal R,\mathcal S, \mathcal L \right)$. This space will contain all the "randomness" of the constructions when we iterate our application described in Section \ref{sectiondescriptionmodele}, except for the input distributions and the location of the exit points $\eta_i$.

Now on the complete infinite tree $\Theta$, we assign the different components of $\mathcal R^{(\theta)}$ and $\mathcal S^{(\theta)}$ as follows: to the edge between $\theta$ and $\theta i$ we associate the values $\mathcal R^{(\theta)}_i$ and $\mathcal S ^{(\theta)}_i$.

These weights in distance and in measure induce the rescaling factors of the vertices, which we define on $\theta$ as the product over the edges on the unique path between the root $\emptyset$ and $\theta$ in $\Theta$. For $\theta=i_1 \ldots i_n$, we define:
\begin{equation} \label{eqn:rtheta}
\mathcal R_\theta:= \prod_{k=1}^n \mathcal R_{i_k}^{(i_1 \ldots i_{k-1})} , \quad \mathcal S_\theta:= \prod_{k=1}^n \mathcal S_{i_k}^{(i_1 \ldots i_{k-1})}.
\end{equation}
In a broader perspective, if for $\theta$  we denote by $C^{(\theta)}$ a function of $(\mathcal R^{(\theta)},\mathcal S^{(\theta)}, \mathcal L^{(\theta)})$, $C_\theta$ will denote the product of all the terms $C_{i_k}^{(i_1 \ldots i_{k-1})}$ on the path between $\emptyset$ and $\theta$ in $\Theta$ with the same notation as in Equation \eqref{eqn:rtheta}.

\section{Proof of existence, uniqueness and convergence statements}
\label{sectionexuniq}

We want to prove our results with respect to the Gromov-Prohorov distance, which boils down to the distribution of the matrix of distances between random points as recalled in Section \ref{matrixdist}. The first and crucial step will be to study the coefficients of this distance matrix, starting with the distance between the root and a random point from which we will be able to study the distance between two random points.

Existence and $L^1$ convergence for both distances are proved in Section \ref{drootunif}. Section \ref{sec:uniq} is devoted to the uniqueness result of Theorem \ref{thmppl} and Section \ref{sec:exist} concludes the proof of this theorem with the existence. Further properties of the fixpoints stated in Theorem \ref{Corpfmoment} and Theorem \ref{thmcviteres} are then proved at the end of the section.

\subsection{$L^1$ convergence of the distance between the root and a random point} \label{drootunif}

The existence and uniqueness results are strongly based on the study of the solutions of Equation \eqref{eqndptunif} and the equivalent one \eqref{eqd2pts} between two randomly chosen points. We will need existence and uniqueness statements for the random variable describing both these distances, for which we have a sufficient condition formalized by Hypotheses \eqref{hypLnonnul} and \eqref{condsufeqnY} for $p=1$.

\label{sectionminimalsolu}

Here we only make the first assumption of Theorem \ref{thmppl}, which is Equation \eqref{hypLnonnul}. Other hypotheses will appear further in the discussion.

Let us recall Equation \eqref{eqndptunif} satisfied by the height of a random point:
\begin{equation*}
    Y \loi \sum\limits_{i \in \Gamma} \underbrace{\mathbf 1_{J \in \Gamma_i} \mathcal R_i^\alpha}_{C_i} Y^{(i)} + \underbrace{\mathcal L (\emptyset,J)}_{Q},
\end{equation*}
which we simplify as
\begin{equation*}
    Y \loi \sum\limits_{i \in \Gamma} C_i Y^{(i)} + Q.
\end{equation*}

We follow the technique from \cite{jelenkovic2012implicit} in order to establish the existence and uniqueness of a solution to Equation \eqref{eqndptunif}. We define the process $(Y_m)_m$ for all $m \in \mathbb N$:
\begin{equation*}
	Y_m:= \sum_{k=0}^m \sum_{\theta \in \Theta_k} C_\theta Q^{(\theta)},
\end{equation*}
where with the same formalism as in Equation \eqref{eqn:rtheta}, for $\theta=i_1 \ldots i_n$, $C_\theta$ is defined as
\begin{equation*}
C_\theta := \prod_{k=1}^n \left( \mathbf 1_{J^{(i_1 \ldots i_{k-1})} \in \Gamma_{i_k}} (\mathcal R_{i_k}^{(i_1 \ldots i_{k-1})} )^\alpha \right). 
\end{equation*}

 We define in the same way $Y_m^\theta$ for all $\theta \in \Theta$. Notice also that $Y_m^\theta$ satisfies the recursive relation:
\begin{equation}
\label{recursionYmonotone}
Y_m^\theta = \sum\limits_{i \in \Gamma} C_i^{(\theta)} Y^{\theta i}_{m-1} + Q^{(\theta)}
\end{equation}

Let $Y$ be the almost sure limit of the sequence $(Y_m)$, well-defined by monotone convergence:
\begin{equation}
\label{lasolutionY}
\begin{split}
Y &:= \lim_{m\rightarrow \infty} Y_m= \sum_{m \in \mathbb N} \sum_{ \genfrac{}{}{0pt}{2}{\theta \in \Theta_m}{\theta=i_1 \ldots i_m} } C_\theta Q^{(\theta)} \\
&= \sum_{m \in \mathbb N} \sum_{ \genfrac{}{}{0pt}{2}{\theta \in \Theta_m}{\theta=i_1 \ldots i_m} } \left( \prod_{k=1}^m C_{i_k}^{(i_1 \ldots i_{k-1})} \right) Q^{(\theta)},
\end{split}
\end{equation}
and for a $\varphi \in \Theta$:
\begin{equation}
Y^\varphi = \sum_{m \in \mathbb N} \sum_{ \genfrac{}{}{0pt}{2}{\theta \in \Theta_m}{\theta=i_1 \ldots i_m} } \left( \prod_{k=1}^m C_{i_k}^{(\varphi i_1 \ldots i_{k-1})} \right) Q^{(\varphi \theta)}.
\end{equation}

By applying monotone convergence to the recursion \eqref{recursionYmonotone}, it is straightforward to verify that $Y^\theta$ solves Equation \eqref{eqndptunif}. Hence we proved the existence of a solution in distribution to this equation. This is indeed the minimal solution to recursion \eqref{eqndptunif} as discussed in \cite{alsmeyer2012fixed}, since for any random variable $Z$ solution to \eqref{eqndptunif} in distribution we can construct a solution $Y$ as above which is stochastically dominated by $Z$ (see the proof of Theorem 3.1 in \cite{alsmeyer2012fixed}). The next discussion will give details about this solution $Y$, its uniqueness and its finite moments, under further conditions.

\begin{lpro} \label{uniquenessYL1}
Assume that Hypotheses \eqref{hypLnonnul} and \eqref{condsufeqnY} hold with $p=1$, then $Y$ constructed above is the unique solution in distribution to \eqref{eqndptunif} with finite mean, and the convergence of the iterated distances $Y_m$ to $Y$ also holds in $L^1$.
\end{lpro}

\begin{proof}
 $ 0<\mathbb E [ \mathcal L (\emptyset,J)] < +\infty$ and $\mathbb E \left[ \sum_i \mathbf 1_{J \in \Gamma_i} \mathcal R_i^\alpha \right] <1$ correspond to the hypotheses of lemmata 4.4 and 4.5 in \cite{jelenkovic2012implicit} with $\beta=1$, hence we can apply their result directly.
\end{proof}

There may exist further non integrable solutions. A broader study of the set of solutions of Equation \eqref{eqndptunif} is done in \cite{alsmeyer2012fixed}.

We also need control on the distance between two independent random points $\zeta,\zeta'$ with distribution $\mu$. From \cite[Lemma 18]{broutin2016self}, we slightly adapt the equation to obtain the recursive equation that must verify $D:=d(\zeta,\zeta')$:
\begin{multline*}
D \loi \sum_{\{i,j\} \subset \Gamma} \mathbf 1_{\{J_1,J_2\}=\{i,j\}} \Bigg( \mathcal L (i,j) + \mathbf 1_{q(i,j) \in \{i,j\}} \mathcal R_{q(i,j)}^\alpha D^{(q(i,j))}   \\  + \sum_{k \in \pi(i,j)} \mathcal R^\alpha_k Y^{(k)}\Bigg),
\end{multline*}
and if we want an equation close to what is considered in \cite{jelenkovic2012implicit} and in the beginning of this section, we group the terms in $D^{(i)}$:
\begin{multline}
\label{eqd2pts}
D \loi \sum_{i \in \Gamma} \left( \sum_{j \in \Gamma_i} \mathbf 1_{\{J_1,J_2\}=\{i,j\}}  \mathcal R_i^\alpha \right) D^{(i)} \\ + \sum_{\{i,j\} \subset \Gamma} \mathbf 1_{\{J_1,J_2\}=\{i,j\}} \left( \mathcal L (i,j) + \sum_{k \in \pi(i,j)} \mathcal R^\alpha_k Y^{(k)} \right),
\end{multline}
which we rewrite
\begin{equation*}
D \loi  \sum_{i \in \Gamma} \tilde{C}_i  D^{(i)} + \tilde{Q}
\end{equation*}
where the $D^{(i)}$s are iid random variables with the same distribution as $D$,  and the $Y^{(i)}$s are iid random variables with the same distribution as $Y$, $J_1,J_2$ are two independent random variables such that $\mathbb P (J_i=j \vert \mathcal R, \mathcal S , \mathcal L)=\mathcal S_j$ representing the index of the (canonical projection $\varphi^\circ$ of the) subset inside of which each of the random points $\zeta, \zeta'$ fall in.

We construct a random variable $\tilde{D}$ as we constructed $Y$ in Equation \eqref{lasolutionY} but with parameters $\tilde{C}, \tilde{Q}$:
\begin{equation*}
\tilde{D}= \lim\limits_{m \rightarrow + \infty} \tilde{D}_m = \sum_{m \in \mathbb N} \sum_{\genfrac{}{}{0pt}{2}{\theta \in \Theta_m}{\theta=i_1 \ldots i_m} } \left( \prod_{k=1}^m \tilde{C}_{i_k}^{(i_1 \ldots i_{k-1})} \right) \tilde{Q}^{(\theta)},
\end{equation*}
and do the same to define $(\tilde{D}^\varphi_m)_m$ and $\tilde{D}^\varphi$ for any $\varphi \in \Theta$. Notice that $\tilde{Q}^{(\theta)}$ is a function of branching parameters of level $\theta$ and of the $Y^{\theta k}$ for all $k \in \Gamma$ constructed in Section \ref{sectionminimalsolu}.

We obtain the same kind of results as for the distance between the root and a random point stated in Proposition \ref{uniquenessYL1}:
\begin{lpro}
\label{uniquenessD}
For any $\theta \in \Theta$, the sequence of iterated random variables $(\tilde{D}_m^\theta)_m$ converges in $L^1$ towards $\tilde{D}^\theta$. The distribution of $\tilde{D}^\theta$ is the unique solution to \eqref{eqd2pts} with finite mean.
\end{lpro}

\begin{proof}
The second term of the equation is in $L^1$ since it is a finite sum of $L^1$ random variables according to our hypothesis on $\mathcal L$ and to Proposition \ref{uniquenessYL1}.

In order to apply the same result from \cite{jelenkovic2012implicit} as it was done in the proof of Proposition \ref{uniquenessYL1}, we also need the following term to be strictly smaller than $1$:
\begin{multline*}
\mathbb E \left[\sum_{i \in \Gamma} \left( \sum_{j \in \Gamma_i} \mathbf 1_{\{\{J_1,J_2\}=\{i,j\}\}}  \mathcal R_i^\alpha \right) \right]\\=
 \mathbb E \left[ \sum_{i \in \Gamma} \left( \mathbf 1_{\{J_1=i \mbox{ and } J_2 \in \Gamma_i, \mbox{ or } J_2=i \mbox{ and } J_1 \in \Gamma_i \}} \mathcal R_i^\alpha \right) \right],
\end{multline*}
the event of which we take the indicator function is a subset of $\{J_1 \in \Gamma_i \mbox{ and } J_2 \in \Gamma_i \}$, itself a subset of $\{J_1 \in \Gamma_i \}$. Hence we can bound it by $\bbE \left[ \sum_i \mathbf 1_{J \in \Gamma_i} \mathcal R_i^\alpha\right]<1$ by Hypothesis \eqref{condsufeqnY} with $p=1$.
\end{proof}

Lastly we introduce another process $(D_m)$ which is linked to the forthcoming sequence of compact rooted measured metric spaces $(\mathcal X_m^\theta)_m$ constructed recursively, presented in Section \ref{sequencerandomspaces} in order to construct a solution to Equation \eqref{eqnppl}. Let $D_m^\theta$ be defined recursively for any $m>0$ and $\theta \in \Theta$ with the $(D_{m-1}^{\theta i},Y_{m-1}^{\theta i})_{ i \in \Gamma}$ and the branching parameters of level $\theta$ just as in Equation \eqref{recursionYmonotone} for $Y_m^\theta$:
\begin{multline}
\label{eqdm}
D_m^\theta := \sum_{i \in \Gamma} \left( \sum_{j \in \Gamma_i} \mathbf 1_{\{J_1,J_2\}=\{i,j\}} ( \mathcal R_i^{(\theta)})^\alpha \right) D_{m-1}^{\theta i} \\ + \sum_{\{i,j\} \subset \Gamma} \mathbf 1_{\{J_1,J_2\}=\{i,j\}} \left( \mathcal L ^{(\theta)} (i,j) + \sum_{k \in \pi(i,j)} (\mathcal R_k^{(\theta)})^\alpha Y_{m-1}^{\theta k} \right).
\end{multline}

Notice here the difference with the sequence $(\tilde{D}_m^\theta)$ that was studied just before,  which satisfies:
\begin{multline}
\label{eqdtildem}
\tilde D_m^\theta := \sum_{i \in \Gamma} \left( \sum_{j \in \Gamma_i} \mathbf 1_{\{J_1,J_2\}=\{i,j\}}  \mathcal R_i^\alpha \right) \tilde D_{m-1}^{\theta i} \\ + \sum_{\{i,j\} \subset \Gamma} \mathbf 1_{\{J_1,J_2\}=\{i,j\}} \left( \mathcal L (i,j) + \sum_{k \in \pi(i,j)} \mathcal R^\alpha_k Y^{\theta k} \right),
\end{multline}
Equation \eqref{eqdm} and Equation \eqref{eqdtildem} differ in the term in $Y$: the sequence we want to study has terms $Y_{m-1}^{\theta k}$ instead of $Y^{\theta k}$ for $\tilde{D}$. We have $Y_{m-1}^{\theta k} \leq Y^{\theta k}$ almost surely for all $k$ since the former is the partial sum of the nonnegative terms $C_\varphi Q^{(\varphi)}$ up to a finite level and the latter is the sum up to infinity. Hence from this remark and from \eqref{eqdm} and \eqref{eqdtildem} we obtain for all $m,\theta$:
\begin{equation*}
D_m^\theta \leq \tilde{D}_m^\theta \mbox{ almost surely}.
\end{equation*}

We are finally able to state the result about the convergence of the sequence $(D_m^\theta)_m$:
\begin{lpro}
\label{limiteD}
Under Hypotheses \eqref{hypLnonnul} and \eqref{condsufeqnY} with $p=1$, for all $\theta \in \Theta$, the sequence $(D_m^\theta)_m$ converges in $L^1$ to a limit $D^\theta$ with the same distribution as $\tilde{D}$.
\end{lpro}

\begin{proof}
For a fixed $\theta$, the sequence $(D_m^\theta)_m$ is increasing, and satisfies for all $m \geq 0$ that almost surely $D_m^\theta \leq \tilde{D}_m^\theta \leq \tilde{D}^\theta$ which is $L^1$. Hence $D_m^\theta$ converges in $L^1$ towards $D^\theta$. 

Moreover, by monotone convergence in Equation \eqref{eqdm}, $D^\theta$ satisfies equation \eqref{eqd2pts} in distribution, and by uniqueness of a $L^1$ solution in distribution we get $D^\theta \loi \tilde{D}^\theta$.
\end{proof}

\subsection{Uniqueness}
\label{sec:uniq}

\begin{lpro} \label{propuniq}
Let $\mathcal X := \left( X,d,\mu,\rho \right)$ and $\mathcal X' := \left( X',d',\mu',\rho' \right)$ be two random variables in $\kgp$, satisfying Equation \eqref{eqnppl} in distribution under \eqref{hypLnonnul} and \eqref{condsufeqnY} for $p=1$.
Then $\mathbb E \big[ \nu^{\mathcal{X}} \big] = \mathbb E \big[ \nu^{\mathcal X'}\big]$.
\end{lpro}

\begin{proof}[Proof of Theorem \ref{thmppl}: uniqueness]

It follows from Proposition \ref{propuniq} and the characterization of Proposition \ref{critereCVmesures}.
\end{proof}

\begin{proof}[Proof of Proposition \ref{propuniq}]
We mostly follow the proof of \cite[Proposition 12]{broutin2016self}, since it is close to the same problem.
Let us denote $M_n$ a standard random variable on $\mathbb M_n$ with distribution $\mathbb E \left[\nu_n^{\mathcal{X}}\right]$ (resp. $M'_n$ with distribution $\mathbb E  \big[ \nu_n^{\mathcal X'}\big]$).

By Proposition \ref{uniquenessYL1}, $M_1$ and $M'_1$ have the same distribution since they are both solutions in distribution of Equation \eqref{eqndptunif} (more precisely their non diagonal coordinates are). We intend to show that for all $n\geq 2$, $M_n$ satisfies a stochastic fixed-point equation which is known to admit at most one solution. If it holds, then $\mathbb E\big[\nu_n^{\mathcal{X}}\big] = \mathbb E \big[\nu_n^{\mathcal{X}'}\big]$ for all $n$, and the theorem is proved.

In order to obtain this result, we will proceed by induction on $n$. Let $A \subset \mathbb M_n$ be a measurable set. From Equation \eqref{eqnppl}, we obtain:
\begin{equation*}
\mathbb E \left[ \nu_n^{\mathcal{X}}\right](A) =\mathbb E \left[ \nu_n^{\psi((\mathcal{X}_i)_i, \mathcal R, \mathcal S, \mathcal L)}(A)\right],
\end{equation*}
and if we expand:
\begin{multline}
\mathbb E \left[ \nu_n^{\mathcal{X}}\right](A) =  \int_{\Sigma_K^2\times \mathcal D_\Gamma }  d\tau(r,s,\ell) \int_{(\kgp)^{\Gamma}}  d\aleph ^{\otimes {\Gamma} }(x_1 ,x_2, \ldots) \\
\mathbb E \left[ \nu_n^{\mathcal{X}^*((x_i)_i, r,s,\ell)}\right](A)
\end{multline}
where $\mathcal X^*$ has distribution $\psi ((x_i)_i,r,s,\ell)$.

We want to describe recursively $M_n$ as a function of where the points fall. We cannot do it straightforwardly, but we introduce another random variable $W_n$ with the same distribution to compute this.

Fix $r,s,\ell$ and let $(\overline{x_i})_i$ be arbitrary representatives of $(x_i)_i$. Sample $\eta_i$ on $\overline{x_i}$ according to $\overline{\mu_i}$, for $i \in \Gamma$. Now write $(\overline{x},\overline{d},\overline{\mu},\overline{\rho})$ constructed from $\Gamma, \alpha, r,s,\ell,(\overline{x_i})_i$, $ (\eta_i)_i$ with the construction described in Section \ref{sectiondescriptionmodele}. To obtain the desired random matrix, we need to generate random independent points $\zeta_j$ for $1 \leq j \leq n$, but we construct them so that we keep the information of the subspace $\overline{x_i}$ in which it falls (or more precisely its image under the canonical projection $\varphi^\circ$). To do so, we first let $\left( J_j\right)_{1 \leq j \leq n}$ be a family of iid random variables on $\Gamma$ with $\mathbb P \left( J_j=i\right)=s_i$ for all $i \in \Gamma$: $J_j$ is the index of the subset in which falls the $j$th random point. For $1 \leq j \leq n$ and $i \in \Gamma$ sample $\tilde{\zeta_{i,j}}$ independently from each others and from the $\eta_i$-s, according to the measure $\overline{\mu_{i}}$ on $\overline{x_{i}}$. Now take $\zeta_j:=\varphi^\circ\left(\tilde{\zeta_{J_j,j}}\right)$ for all $1 \leq j \leq n$, and $\zeta_0:= \overline{\rho}$. By construction, the matrix $W_n:=\big( d(\zeta_j,\zeta_k) \big)_{0 \leq j,k \leq n}$ has the desired distribution $\mathbb E \left[ \nu_n^{\mathcal X^* ((x_i)_i,r,s,\ell)} \right]$.

The random matrix $W_n$ is more handy, since we can decompose it by looking in which subspace $\varphi^\circ(\overline{x_i})$ do the $\zeta_j$ points fall by the value of each one of the $J_j$. In each space $\overline{x_i}$, we may compare the distances inside this space between all the $\zeta_j$ falling in $\varphi^\circ(\overline{x_i})$ possibly in addition to the root $\overline{\rho_i}$ and the branchpoint $\eta_i$.

For all $i \in \Gamma$ and $1 \leq m \leq n$ we write $Y^{(m)}_{\overline{x_i}}$ for a random variable with distribution $\nu_m^{\overline{x_i}}$. Let $\mathbf i=(i_1 , \ldots , i_n) \in \Gamma^n$ be a $n$-tuple of elements of $\Gamma$, and let $\mathcal E_{\mathbf i}$ be the event that $J_j=i_j$ for all $1 \leq j \leq n$. On this event, $W_n$ has the same distribution as a linear combination of copies of the $Y^{(m)}_{\overline{x_i}}$, or more precisely their image via some deterministic linear operators depending on the fixed parameters $(r,s,\ell)$. In order to obtain the uniqueness result, the only terms that need to be computed are the ones displaying copies of th $Y^{(n)}_{\overline{x_i}}$, thus the next part of the proof is the precise description of $W_n$ under the events $\mathcal E_{\mathbf i}$ on which such random matrices of size $(n+1)$ appear.

We will use the linear operators $F_j:\mathbb R \longrightarrow  \mathbb M_n $ defined for indices $0 \leq j \leq n$, such that for any $a \in \mathbb R$, $F_j(a)$ is defined as the matrix:
\[
\begin{tikzpicture}[decoration=brace]\matrix (m) [matrix of math nodes,left delimiter=(,right delimiter={)}] {
		0&\ldots&0&a&&&\\
		\vdots&\ddots&\vdots&\vdots&&(0)&\\
		0&\ldots&0&a&&&\\
		a&\dots&a&0&a&\dots&a\\
		&&&a&&&\\
		&(0)&&\vdots&&(0)&\\
		&&&a&&&\\
	};
    \draw[decorate,transform canvas={xshift=-1.5em},thick] (m-3-1.south west) -- node[left=2pt] {$j$} (m-1-1.north west);
    \draw[decorate,transform canvas={yshift=0.5em},thick] (m-1-1.north west) -- node[above=2pt] {$j$} (m-1-3.north east);
\end{tikzpicture}.
\]
It will be mostly used to add distances outside the subset in which most of the points fall.

The distance between $n+1$ points inside a single subspace is computed in the conditional expression of $W_n$ when most of the points fall in the same subspace. In fact, in happens only on the events $\mathcal E_{\mathbf i}$ described as follows: for indexes $i \in \Gamma$ and $i' \in \Gamma_{i}$, there is a unique $1\leq k\leq n$ such that $i_k=i'$ and we have $i_j=i$ for all $j\neq k$. Indeed, it combines $4$ subcases briefly described as follows:
\begin{enumerate}
\item \label{uniquenesscas1} If $\mathbf i=(\emptyset , \ldots , \emptyset)$, then all the distances between the $\zeta_j$ points are all measured within $\overline{x_\emptyset}$. Conditionally on the event $\mathcal E_\mathbf i$, $W_n$ is distributed as $r_\emptyset^\alpha Y^{(n)}_{\overline{x_\emptyset}}$.
\item \label{uniquenesscas2} If for a $i \in \Gamma-\{\emptyset\}$ we have $\mathbf i=(i , \ldots , i )$, then all the distances between the points $\zeta_1 , \ldots , \zeta_n$ are measured within $\overline{x_i}$. For $1 \leq j \leq n$, the distance between $\zeta_j$ and $\zeta_0=\overline{\rho}$ can be decomposed into two parts: on the one hand the distance between the canonical projection of $\zeta_j$ $\varpi^\circ(\overline{\rho_i})$ inside $\overline{x_i}$, and on the other hand $\overline{d}(\overline{\rho},\varphi^\circ(\rho_i))$ which is equal to the sum of the distances inside all the crossed subspaces on the geodesic between $\zeta_0$ and $\varphi^\circ(\rho_i)$, which are all the $k \in \Gamma$ such that $i \in \Gamma_k$ except for $i$, plus the length $\ell$. Hence $W_n$ is distributed as:
\[
        r_i^\alpha Y^{(n)}_{\overline{x_i}} + F_0 \left( \sum\limits_{k\neq i:i \in \Gamma_k} r_{k}^\alpha \left(Y^{(1)}_{\overline{x_{k}}}\right)_{0,1} + \ell(\emptyset,i) \right)
\]
where the different random variables $Y$ are independent, and we point out that $\left(Y^{(1)}_{\overline{x_{k}}}\right)_{0,1}$ is the matrix coefficient corresponding to the distance inside $\overline{x_{k}}$ between $\overline{\rho_k}$ and $\eta_k$.

\item \label{uniquenesscas3} If for $i \in \Gamma-\{\emptyset\}$ and $1 \leq k \leq n$ we have $i_j=\emptyset$ for all $j \neq k$ and $i_k=i$. The distances $d(\zeta_j,\zeta_m)$ are all computed within $\overline{x_\emptyset}$ for $j,m \neq k$ (includind $0$). What remains is the distance between $\zeta_k$ and all the others $\zeta_j$ inside $\overline{x_\emptyset}$. This distance can be decomposed in the same way as in the case \ref{uniquenesscas2} into  the distance needed to go out of $\overline{x_\emptyset}$ via $\varphi^\circ(\eta_\emptyset)$, and all the distance between the images via $\varphi^\circ$ of $\eta_\emptyset$ and $\zeta_k$. Then the matrix $W_n$ is distributed as:
\[
r_\emptyset^\alpha Y^{(n)}_{\overline{x_\emptyset}} + F_k \left( \sum\limits_{p \neq 0: i \in \Gamma_p} r_{p}^\alpha Y^{(1)}_{\overline{x_{p}}} + \ell(\emptyset,i) \right)
\]
\item \label{uniquenesscas4} There is one last case in which a random variable $Y^{(n)}_{\overline{x_i}}$ appears in the description of the law of $W_n$. If there are two distinct elements $i,i'$ of $\Gamma-\{\emptyset\}$ such that $i' \in \Gamma_i$, there is an index $k \in \{1 , \ldots , n \}$ such that $J_k=i'$, and $J_j=i$ for all $j \neq k$. Conditionally on the event $\mathcal E_{\mathbf i}$, $W_n$ is distributed as:
\begin{multline*}
r_i^\alpha Y^{(n)}_{\overline{x_i}} + F_0 \left( \sum\limits_{p \neq i:i \in \Gamma_p} r_{p}^\alpha Y^{(1)}_{\overline{x_{p}}} + \ell(\emptyset,i) \right)\\
+F_j \left( \sum\limits_{p \neq i: p \in \Gamma_{i} , i' \in \Gamma_p } r_{p}^\alpha Y^{(1)}_{\overline{x_{p}}} + \ell(i,i') \right)
\end{multline*}
which is somehow a combination of cases \ref{uniquenesscas2} and \ref{uniquenesscas3}. $Y^{(n)}_{\overline{x_i}}$ appears because what is computed is the distance between the $n-1$ random points $\zeta_j$ falling into $\overline{x_i}$, the root of this subspace $\overline{\rho_i}$, and the branchpoint $\eta_i$ which plays the same role as if the last point $\zeta_k$ also fell in $\overline{x_i}$.
\end{enumerate}

In any other case than the one we described before, the conditional distribution of $W_n$ depends only on the branching parameters and on copies of the $Y_{\overline{x_i}}^{(k)}$ for some $k<n$. To be more precise, the maximal value of $k$ appearing in a $Y_{\overline{x_i}}^{(k)}$ is the maximum amount of random points falling inside the same $\overline{x_i}$, plus $1$ whenever at least one other point falls in $\Gamma_i-\{i\}$ (so that the branchpoint $\eta_i$ intervenes as in the same way as in cases \ref{uniquenesscas3} and \ref{uniquenesscas4}). Hence we described the only situations where $k=n$.

Note that the other conditional distributions of $W_n$ can be written deterministically in every case, however it will not be necessary to write them all for the remaining of the proof.

If we combine everything we considered previously, the random matrix $M_n$ is a solution of the following distributional equation:
\begin{equation} \label{perpetuity}
M_n \loi A_n M_n + B_n
\end{equation}
where
\begin{equation*}
A_n= \sum_{i \in \Gamma} \mathbf 1_{\big\{n-1 \: \zeta_l \mbox{ fall into } \overline{x_i} \mbox{ and the last one falls into } \overline{x_{i'}}, \: i' \in \Gamma_i \big\}} \mathcal R_i^\alpha,
\end{equation*}
which is in $[0,1]$ almost surely, and $B_n \in \mathbb M_n$. Moreover, $D_n, (A_n,B_n)$ are independent random variables.

We do not develop the expression of $B_n$ here, but refer to the expression of $V_n$ in the proof of Proposition 12 in \cite{broutin2016self}: the expression of our $B_n$ is a sum of images of independent copies of $M_m$ for $m<n$, independent of $(\mathcal R,\mathcal S,\mathcal L)$, by linear operators deterministic in the branching parameters $\mathcal R,\mathcal S$, to which are added at the right coordinates the branching lengths $\ell(i,j)$.

A random variable such as $M_n$ satisfying a fixed-point equation as Equation \eqref{perpetuity} is called a \emph{perpetuity}. If follows from classical results on perpetuities, for instance Theorem 1.5 in \cite{vervaat1979stochastic}, that this fixed-point equation has at most one solution in distribution. To conclude on the heredity, remark that $M'_n$ satisfies an equation $M'_n \loi A'_n M'_n + B'_n$ as Equation \eqref{perpetuity} with $A'_n=A_n$, and $B_n$ is replaced by $B'_n$ simply by changing every copy of $M_m$ for $m<n$ by a copy of $M'_m$ instead still independent from each other and from $(\mathcal R,\mathcal S,\mathcal L)$. By the induction hypothesis, for all $m<n$, $M_m$ and $M'_m$ have the same distribution. It follows that $M_n$ and $M'_n$ have the same distribution which concludes the proof.
\end{proof}

\subsection{Existence of a solution: an iterative construction}
\label{sec:exist}

In this section, we construct a solution to Equation \eqref{eqnppl}. To do so, we  iterate the application $\Psi$ starting from singletons. In order to prove the convergence of this sequence of random variables taking values in $\kgp$, we prove a convergence of the height of a random point as introduced in Section \ref{sectionminimalsolu}, followed by the study of the distance between two randomly chosen points.
Together, these convergences in $L^1$ entail the weak convergence of the distance matrix distribution. The last ingredient we need is the tightness of the sequence in order to apply Proposition \ref{critereCVmesures} $(ii)$.

Following \cite{greven2009convergence}, let us define on an isometry class of rooted measured metric space $\mathfrak X =(X,d,\mu,\rho)$ in $\kgp$, the \emph{distance distribution} and the \emph{modulus of mass distribution} respectively as:
\[ \begin{array}{rl}
w_{\mathfrak X} & := d_* \mu^{\otimes 2} \\
v_\delta (\mathfrak X) & := \inf \left\{ \epsilon>0 : \mu \left\{x \in X : \mu(B_\epsilon(x)) \leq \delta \right\} \leq \epsilon \right\}, \mbox{ for a } \delta>0.
\end{array}
\]
For a probability measure $\mathbf P$ on $\kgp$ we define in the same way as in Equation \eqref{nudelta} the annealed versions of both the \emph{distance distribution} $w_{\mathbf P}$ and the \emph{modulus of mass distribution} $v_\delta(\mathbf P)$.

To address tightness in $\kgp$, we will use the characterization given by \cite[Theorem 3]{greven2009convergence}: a sequence of probability measures $(\mathbf P_m)_{m \geq 1}$ in $\mathcal M_1 (\kgp)$ is tight if and only if the family $\left\{  w_{\mathbf P_m} : m \in \mathbb N \right\}$ is tight in $\mathcal M_1 \left( \mathbb R \right)$ and
\begin{equation}\label{tightnessgrev}
\text{ For all $\epsilon>0$, there exists a $\delta>0$ such that } \sup_{m \in \mathbb N} v_\delta(\mathbf P_m) <\epsilon.
\end{equation}

We will also use the notations $\mathbf E [w_{\mathcal X}]$ and $\mathbf E [v_\delta (\mathcal X)]$ when needed.

\paragraph{Our sequence of random variables} \label{sequencerandomspaces} We define $\mathcal X_m^\theta := \left( X_m^\theta, d_m^\theta, \mu_m^\theta, \rho_m^\theta \right)$ recursively as follows: let $\mathcal X_0^\theta = \{ \rho^\theta_0 \}$ for all $\theta$ (the initial spaces are singletons), and for $m \geq 1$, define
\begin{equation*}
\mathcal X_m^\theta := \psi \left( \left( \mathcal X^{\theta i}_{m-1} \right)_{i \in \Gamma}, \mathcal R^{(\theta)}, \mathcal S^{(\theta)} , \mathcal L^{(\theta)} \right).
\end{equation*}

$\mathcal X_m^\theta$ is a random variable in $\kgp$. We denote by $\mathbf P_m^\theta$ its distribution. The random measure $\mu_m^\theta$ on $X_m^\theta$ gives weight $S^{(\theta)}_\varphi$ to each $\rho_0^{\theta \varphi}$ (or rather its image via the successive canonical projections) for any $\varphi \in \Theta_m$. We do not explicitely write the distance $d_m^\theta$ between any two points, but remark that $D_m^\theta$ which recurrence is described in Equation \eqref{eqdm} is exactly the distance between two random points independently chosen both wih distribution $\mu_m^\theta$. Similarly $Y_m^\theta$ is the distance between $\rho_m^\theta$ and a random point on $\mathcal X_m^\theta$ wit distribution $\mu_m^\theta$.

At a fixed $\theta$, $(\mathcal X^\theta_m)_{m \in \mathbb N}$ is a sequence of random compact rooted measured metric spaces.

This section will focus on the proof of the following proposition which proves immediately the existence in Theorem \ref{thmppl}:
\begin{lpro} \label{lalimite}
For any $\theta \in \Theta$, the sequence $( \mathcal X_m^\theta)_m$ converges in distribution in $\kgp$ to a random variable $\mathcal X^\theta$. For any $\theta \in \Theta$, we have:
\begin{equation} \label{eqntheta}
\mathcal X^\theta \loi \psi \left( \left( \mathcal X^{\theta i}\right)_{i \in \Gamma} , \mathcal R^\theta , \mathcal S^\theta , \mathcal L^\theta \right).
\end{equation}
\end{lpro}

We split the proof of Proposition \ref{lalimite} into the two following lemmata, the first one states the weak convergence of the random distance matrices of the sequence of random metric spaces considered, and the second one deals with the tightness of the sequence of probability distributions. Together they imply Proposition \ref{lalimite} by applying Proposition \ref{critereCVmesures}.

\begin{llem}
The sequence $(\mathbf E_m^\theta [\nu^{\mathcal X^\theta_m}])_{m \geq 0}$ converges weakly.
\end{llem}

\begin{proof}
The convergences of $(D_m^\theta)$ and $(Y_m^\theta)$ both hold in $L^1$ from Proposition \ref{uniquenessYL1} and Proposition \ref{limiteD}. Hence the joint convergence of distance matrices of any finite size are also true in $L^1$, and the weak convergence of the distance matrices follows.
\end{proof}

From now on, we consider the case $\theta=\emptyset$ for more clarity in the expressions.

\begin{llem} \label{lemtight}
The family of distributions $\left( \mathbf P_m^\theta\right)_{m \geq 0}$ is tight in the set of probability distributions on $\kgp$.
\end{llem}

\begin{proof}
As we said, we use the characterization from \cite[Theorem 3]{greven2009convergence}. The first part of the criterion is quickly obtained, since $w_{\mathbf P_m}$ is the distribution of $D_m$. We have shown that this sequence is bounded in $L^1$ in Proposition \ref{limiteD}, hence the tightness.

For the second part of the criterion recalled in Equation \eqref{tightnessgrev}, we have to show that for all $\epsilon>0$ there exists a $\delta>0$ such that
\begin{equation}
\limsup_{m \in \mathbb N} \mathbf E _m \left[ \mu_m \left\{ x \in \mathcal X_m \; : \; \mu_m(B_\epsilon(x)) \leq \delta \right\} \right] \leq \epsilon,
\end{equation}
as discussed in \cite[Remark 3.2]{greven2009convergence}.

Let $\epsilon>0$ be fixed, and take $m \in \mathbb N$. We introduce another index $M \leq m$ and separate the space $X_m$ at that level. We omit to specify the canonical projection $\varphi^\circ$ at each level of the recursion to simplify the notations.
\begin{equation*}
\begin{split}
v_\delta(\mathcal X_m)& = \mu_m \left\{ x \in X_m \; : \; \mu_m(B_\epsilon(x))\leq \delta \right\} \\
&= \mu_m \left\{ x \in \! \bigsqcup_{\varphi \in \Theta_M} X_{m-M}^\varphi \; : \; \mu_m(B_\epsilon(x))\leq \delta \right\}.
\end{split}
\end{equation*}
In the last expression, the union is disjoint. On every subspace $X_{m-M}^\varphi$, we decompose one more time according to the distance to the points $\rho_{m-M}^\varphi$ being greater or lesser than a constant:
\begin{multline*}
v_\delta(\mathcal X_m)=\mu_m \left\{ x \in \bigsqcup_{\varphi \in \Theta_M} \left( X_{m-M}^\varphi\cap B_{\epsilon/2}(\rho_{m-M}^\varphi) \right) \; : \; \mu_m(B_\epsilon(x))\leq \delta \right\} \\ + \mu_m \left\{ x \in \bigsqcup_{\varphi \in \Theta_M} \left( X_{m-M}^\varphi \backslash B_{\epsilon/2}(\rho_{m-M}^\varphi) \right) \; : \; \mu_m(B_\epsilon(x))\leq \delta \right\},
\end{multline*}
and condense as $v_\delta(\mathcal X_m)= A+B$.

Note that for $\eta>0$, $B_\eta(y)$ is the ball of radius $\eta$ around $y$, and it will always be considered in the space $\mathcal X_m$ endowed with its distance $d_m$.

\paragraph{Bound on $\mathbf E_m [B]$} Here the upper bound comes from the fact that after enough iterations, the distances we add are lesser than any constant with high probability. Hence all the mass stays relatively close to the point $\rho_{m-M}^\varphi$ in the corresponding subspace.
\begin{equation*}
\begin{split}
B &=\mu_m \left\{ x \in \bigsqcup_{\varphi \in \Theta_M} \left( X_{m-M}^\varphi\backslash B_{\epsilon/2}(\rho_{m-M}^\varphi) \right) \; : \; \mu_m(B_\epsilon(x))\leq \delta \right\} \\
&= \sum_{\varphi \in \Theta_M} \mathcal S_\varphi \mu_{m-M}^\varphi \left\{ x \in  \left( X_{m-M}^\varphi\backslash B_{\epsilon/2}(\rho_{m-M}^\varphi) \right) \; : \; \mu_m(B_\epsilon(x))\leq \delta \right\}.
\end{split}
\end{equation*}
On a fixed $\varphi \in \Theta_M$, we have almost surely
\begin{equation*}
\begin{split}
&\mu_{m-M}^\varphi \left\{ x \in  \left( X_{m-M}^\varphi\backslash B_{\epsilon/2}(\rho_{m-M}^\varphi) \right) \; : \; \mu_m(B_\epsilon(x))\leq \delta \right\} \\
&\leq \mu_{m-M}^\varphi \left\{ X_{m-M}^\varphi\backslash B_{\epsilon/2}(\rho_{m-M}^\varphi)  \right\} \\
&= \mu_{m-M}^\varphi \left\{x \in X_{m-M}^\varphi\; : \; d_m(x,\rho_{m-M}^\varphi)>\epsilon/2  \right\}.
\end{split}
\end{equation*}

Restricted to $X_{m-M}^\varphi$, $d_m= (\mathcal R_\varphi)^\alpha d_{m-M}^\varphi$ where $\mathcal R_\varphi$ is a random variable almost surely in $(0,1)$. We first compute the conditional expectation with respect to the random variables $(\mathcal R^{(\theta)} , \mathcal S ^{(\theta)}, \mathcal L ^{(\theta)})_{\vert \theta \vert < M}$. We denote by $\mathcal F_M$ the $\sigma$-algebra generated by this family of random variables. Since for any $\varphi \in \Theta_M$, $(\mathcal S_\varphi,\mathcal R_\varphi)$ is measurable with respect to $\mathcal F_M$,
\begin{multline*}
\mathbf E_m \left[ B \vert \mathcal F_M \right] \\ \leq \sum_{\varphi \in \Theta_M} \mathcal S_\varphi \mathbf E_m \left[ \mu_{m-M}^\varphi \left\{ x \in  X_{m-M}^\varphi \; : \; (\mathcal R_\varphi)^\alpha d_{m-M}^\varphi(x,\rho_{m-M}^\varphi) > \epsilon/2 \right\} \vert \mathcal F_M \right].
\end{multline*}

For a fixed $\varphi \in \Theta_M$, note that
\begin{multline*}
\mathbf E_m \left[\mu_{m-M}^\varphi \left\{ x \in  X_{m-M}^\varphi \; : \; (\mathcal R_\varphi)^\alpha d_{m-M}^\varphi(x,\rho_{m-M}^\varphi) > \epsilon/2 \right\} \vert \mathcal F_M \right] \\
= \mathbf P_m \left( Y_{m-M}^\varphi > \frac{\epsilon}{2 \left( \mathcal R_\varphi \right)^\alpha} \vert \mathcal F_M \right),
\end{multline*}
which we bound using Markov's inequality by
\begin{equation*}
\mathbf E_m \left[ Y_{m-M}^\varphi \right] \frac{2 \left( \mathcal R_\varphi \right)^\alpha}{\epsilon},
\end{equation*}
since $Y^\varphi_{m-M}$ is independent from $\mathcal F_M$. Now we compute the total expectation of $B$
\begin{equation*}
\mathbf E_m \left[ B \right] \leq \sum\limits_{\varphi \in \Theta_M} \frac{2}{\epsilon} \mathbf E_m \left[Y_{m-M}^\varphi \right] \mathbf E_m \left[ \mathcal S_\varphi ( \mathcal R_\varphi)^\alpha \right]
\end{equation*}

By the construction of the sequences of $(Y_m^\theta)_m$ for all $\theta \in \Theta$ in Section \ref{sectionminimalsolu}, we have $Y_{m-M}^\varphi \leq Y^\varphi$, and the latter random variable is an independent copy of $Y^\emptyset$. Thus $\mathbf E_m \left[ Y_{m-M}^\varphi \right] \leq \mathbb E \left[ Y \right] < +\infty$ since $Y$ is $L^1$. As a consequence,
\begin{equation}
\label{boundBlast}
\mathbf E_m \left[ B \right] \leq \frac{2 \mathbb E \left[ Y \right]}{\epsilon}\mathbf E_m \left[ \sum\limits_{\varphi \in \Theta_M} \mathcal S_\varphi (\mathcal R_\varphi)^\alpha \right].
\end{equation}

Remember Hypothesis \eqref{condsufeqnY} with $p=1$ that $\varrho:= \mathbb E \left[ \sum_{i \in \Gamma} \mathbf 1_{J \in \Gamma_i} \mathcal R_i^\alpha \right]<1$, where $J$ is a random variable on $\Gamma$ such that for all $k \in \Gamma$, $\mathbb P \left(J=k \vert \mathcal R,\mathcal S , \mathcal L \right)=\mathcal S_k$. We have the immediate bound
\begin{equation*}
\mathbb E \left[ \sum\limits_{i \in \Gamma} \mathcal S_i (\mathcal R_i)^\alpha \right] \leq \mathbb E \left[ \sum\limits_{i \in \Gamma} \left(\sum_{j \in \Gamma_i} \mathcal S_j \right) (\mathcal R_i)^\alpha \right] = \mathbb E \left[ \sum\limits_{i \in \Gamma} \mathbf 1_{J \in \Gamma_i} (\mathcal R_i)^\alpha \right]=\varrho,
\end{equation*}
and for any $m \in \mathbb N$ and $\varphi \in \Theta$ such that $\vert \varphi \vert \leq m$, note that because the $\left(\mathcal S^{(\theta)},\mathcal R^{(\theta)} \right)_\theta$ are iid copies of $\left(\mathcal S,\mathcal R\right)$
\begin{equation*}
\mathbf E_m \left[ \sum\limits_{i \in \Gamma} \mathcal S_i^{(\varphi)} (\mathcal R_i^{(\varphi)})^\alpha \right] = \mathbb E \left[ \sum\limits_{i \in \Gamma} \mathcal S_i (\mathcal R_i)^\alpha \right].
\end{equation*}
Hence by induction
\begin{equation}\label{rhom}
 \mathbf E_m \left[ \sum\limits_{\varphi \in \Theta_M} \mathcal S_\varphi ( \mathcal R_\varphi)^\alpha \right] \leq \varrho^M.
\end{equation}

The bound on $B$ is obtained if we fix $M$ large enough such that $\frac{2 \mathbb E \left[ Y \right]}{\epsilon} \varrho^M<\frac{\epsilon}{2}$.

\paragraph{Bound on $\mathbf E_m[A]$} $M$ is fixed from the upper bound on $\mathbf E_m[B]$. The remaining term to bound (in mean) is:
\begin{equation*}
A=\mu_m \left\{ x \in \bigsqcup_{\varphi \in \Theta_M} \left( X_{m-M}^\varphi \cap B_{\epsilon/2}(\rho_{m-M}^\varphi) \right) \; : \; \mu_m \left( B_\epsilon (x) \right) \leq \delta \right\}
\end{equation*}
for some $\delta>0$ uniformly in $m\geq M$.

Let us first bound $A$ almost surely by a simpler expression. Since $B_{\epsilon/2}(\rho_{m-M}^\varphi) \subset B_{\epsilon}(x)$ for any $x \in X_{m-M}^\varphi \cap B_{\epsilon/2}(\rho_{m-M}^\varphi)$, we have the first almost sure bound
\begin{equation*}
\begin{split}
A &\leq \mu_m \left\{ x \in \bigsqcup_{\varphi \in \Theta_M} \left( X_{m-M}^\varphi \cap B_{\epsilon/2}(\rho_{m-M}^\varphi) \right) \; : \; \mu_m \left( B_{\epsilon/2}(\rho_{m-M}^\varphi) \right) \leq \delta \right\} \\
 &= \sum_{\varphi \in \Theta_M} \mathcal S_\varphi \mu_{m-M}^\varphi \left\{x \in B_{\epsilon/2}(\rho_{m-M}^\varphi) \; : \; \mu_m(B_{\epsilon/2}(\rho_{m-M}^\varphi)) \leq \delta \right\} \\
&\leq \sum_{\varphi \in \Theta_M} \mathcal S_\varphi \mu_{m-M}^\varphi \left\{x \in B_{\epsilon/2}(\rho_{m-M}^\varphi) \; : \; \mathcal S_\varphi \mu_{m-M}^\varphi(B_{\epsilon/2}(\rho_{m-M}^\varphi)) \leq \delta \right\} \\
&= \sum_{\varphi \in \Theta_M} \mathcal S_\varphi \mu_{m-M}^\varphi\left( B_{\epsilon/2}(\rho_{m-M}^\varphi) \right) \mathbf 1_{ \mathcal S_\varphi \mu_{m-M}^\varphi(B_{\epsilon/2}(\rho_{m-M}^\varphi)) \leq \delta}.
\end{split}
\end{equation*}

Now when we look at the expectation, and split each term of the sum according to the size of $\mu_{m-M}^\varphi(B_{\epsilon/2}(\rho_{m-M}^\varphi))$ being larger or smaller than $\varepsilon/4$ to bound $\mathbf E_m\left[A \right]$ by
\begin{multline*}\mathbf E_m \left[ \sum_{\varphi \in \Theta_M} \mathcal S_\varphi \mu_{m'}^\varphi\left( B_{\epsilon/2}(\rho_{m'}^\varphi) \right) \mathbf 1_{ \mathcal S_\varphi \mu_{m'}^\varphi(B_{\epsilon/2}(\rho_{m'}^\varphi)) \leq \delta} \mathbf 1_{\mu_{m'}^\varphi(B_{\epsilon/2}(\rho_{m'}^\varphi)) \leq  \varepsilon /4 } \right] \\
+ \mathbf E_m \left[ \sum_{\varphi \in \Theta_M} \mathcal S_\varphi \mu_{m'}^\varphi\left( B_{\epsilon/2}(\rho_{m'}^\varphi) \right) \mathbf 1_{ \mathcal S_\varphi \mu_{m'}^\varphi(B_{\epsilon/2}(\rho_{m'}^\varphi)) \leq \delta}  \mathbf 1_{\mu_{m'}^\varphi(B_{\epsilon/2}(\rho_{m'}^\varphi)) \geq  \varepsilon /4 }\right],
\end{multline*}
where we noted $m':=m-M$.

The first term is bounded by $\varepsilon/4$, and the second one by:
\begin{equation*}
\mathbf E_m \left[\sum_{\varphi \in \Theta_M} \mathcal S_\varphi \mathbf 1_{\mathcal S_\varphi \leq  4\delta /\varepsilon} \right].
\end{equation*}

We follow \cite[Section 1.2.3]{MR2253162} for this bound. With the setting of Bertoin, we can see the family $(\mathcal S_\varphi)_{n \in \mathbb{N}, \varphi \in \Theta_n}$ as the genealogical tree of a self-similar fragmentation. It defines a probability measure $\mathbb P^*$ describing the distribution of a randomly tagged branch. With \cite[Lemma 1.4]{MR2253162}, if we note by $\bar{\mathcal S_k}$ a random variable describing the size of the subspace of level $k$ in which the randomly tagged particle falls in, we get that:
\begin{equation*}
\mathbf E_m \left[\sum_{\varphi \in \Theta_M} \mathcal S_\varphi \mathbf 1_{\mathcal S_\varphi \leq  4\delta /\varepsilon} \right] = \mathbb E^* \left[\mathbf 1_{\bar{\mathcal S_M} \leq  4\delta /\varepsilon} \right]=\mathbb P^* \left(\bar{\mathcal S_M} \leq  4\delta /\varepsilon \right).
\end{equation*}

\cite[Proposition 1.6]{MR2253162} gives us a way to exploit this expression. Write:
\begin{equation*}
\mathbb P^* \left(\bar{\mathcal S_M} \leq  4\delta /\varepsilon \right)=\mathbb P^* \left( - \log \bar{\mathcal S_M} \geq  \tau \right) \mbox{ for } \tau=- \log (4\delta /\varepsilon),
\end{equation*}
and $(- \log \bar{\mathcal S_k})_{k \in \mathbb N}$ is a random walk on $\mathbb R+$ of which we can specify the step distribution. Now since $\bar{\mathcal S_1}>0$ almost surely, the random variable  $- \log \bar{\mathcal S_M}$ is finite almost surely. Let $q>0$ be an $\varepsilon/4$-quantile of this random variable, such that $\mathbb P^* \left( - \log \bar{\mathcal S_M} \geq  q \right) \leq \varepsilon/4$. $\delta$ chosen so that $\delta < \varepsilon e^{-q} /4$ yields the desired bound for the last term.

We proved the bound $\limsup_{m \in \mathbb N} \mathbf E_m \left[v_\delta(\mathcal X_m)\right] \leq \varepsilon$. The sequence of random metric spaces verifies both hypotheses of \cite[Theorem 3]{greven2009convergence}, hence the tightness.
\end{proof}

\subsection{Proof of Theorem \ref{Corpfmoment}}

\begin{proof}[Proof of Theorem \ref{Corpfmoment}]
\emph{(i)} is a straightforward application of \cite[Lemma 4.4]{jelenkovic2012implicit}. The fix-point $\eta$ is constructed as previously, with the distance to a random point having finite $p$th moment.

\emph{(ii)}
Let $m\in \mathbb N$ be chosen such that 
\begin{equation}\label{choixm}
m \alpha>1.
\end{equation}
Since the random variables $\mathcal R_i$ are all lesser than $1$, we have
\begin{equation}
\mathbb E \left[\sum_{i \in \Gamma} \mathcal R_i^{m\alpha} \right]< 1
\end{equation}

We consider the sequence of iterated spaces $(\mathcal X_n)_n$ from the construction of the limit space $\mathcal X$ described in Section \ref{sequencerandomspaces}.

We want constants $r<1$ and $C>0$ such that for all $n \in \mathbb N$:
\begin{equation}
\bbE \left[ \left( h(\mathcal X_{n+1})\right)^m \right]^{\frac{1}{m}} \leq r \bbE \left[(h\left(\mathcal X_{n})\right)^m \right]^{\frac{1}{m}} + C.
\end{equation}
With that in hand, the limit converges when $n$ tends to infinity, which ensures that  $\bbE\left[\left(h(\mathcal X)\right)^m\right]<+\infty$.

Now let $n\geq 0$. By the construction of the space $\mathcal X_{n+1}$ from the spaces $(\mathcal X_n^i)_{i \in \Gamma}$ and parameters $(\mathcal R,\mathcal S,\mathcal L)$, we have the relation concerning the height:
\begin{equation}\label{eq:basehm}
h(\mathcal X_{n+1}) = \sup\limits_{i \in \Gamma} \left\{ \mathcal R_i^\alpha h(\mathcal X^i_{n}) + \mathcal L (\emptyset,i) + \sum\limits_{j \neq i : i \in \Gamma_j} \mathcal R_j^\alpha d_n^j \left(\rho_n^j,\eta_n^j\right) \right\}
\end{equation}

Taking the $m$-th moment from Equation \eqref{eq:basehm} and maximizing each term:
\begin{multline}\label{eq:hmsum}
\bbE \left[ h^m(\mathcal X_{n+1})\right]^{\frac{1}{m}} \\ \leq \bbE \left[ \left\{\sup\limits_{i \in \Gamma} \mathcal R_i^\alpha h(\mathcal X^i_{n})  + \sup\limits_{i \in \Gamma}\mathcal L (\emptyset,i) + \sup\limits_{i \in \Gamma} \sum\limits_{j \neq i : i \in \Gamma_j} \mathcal R_j^\alpha d_n^j \left(\rho_n^j,\eta_n^j\right) \right\}^m \right]^{\frac{1}{m}}
\end{multline}

Then from Minkowski's inequality:
\begin{multline}
\bbE \left[ h^m(\mathcal X_{n+1})\right]^{\frac{1}{m}} \leq \bbE \left[\sup\limits_{i \in \Gamma}  \mathcal R_i^{m \alpha} h^m(\mathcal X^i_{n})\right]^\frac{1}{m}  +\bbE \left[ \sup\limits_{i \in \Gamma}\mathcal L ^m(\emptyset,i) \right]^\frac{1}{m} \\
 +\bbE \left[ \sup\limits_{i \in \Gamma}\left( \sum\limits_{j \neq i : i \in \Gamma_j} \mathcal R_j^\alpha d_n^j \left(\rho_n^j,\eta_n^j\right)\right)^m \right]^\frac{1}{m}
\end{multline}
In the first term, we can bound the supremum by the sum over all $i \in \Gamma$:
\begin{equation}
\begin{split}
\bbE \left[\sup\limits_{i \in \Gamma}  \mathcal R_i^{m \alpha} h^m (\mathcal X^i_{n})\right]^{\frac{1}{m}} &\leq \left( \bbE \left[\sum\limits_{i \in \Gamma}  \mathcal R_i^{m \alpha} h^m (\mathcal X^i_{n})\right] \right)^{\frac{1}{m}} \\
&\leq \bbE \left[\sum\limits_{i \in \Gamma}  \mathcal R_i^{m \alpha} \right] ^{\frac{1}{m}} \bbE \left[ h^m (\mathcal X_{n})\right]^{\frac{1}{m}}.
\end{split}
\end{equation}

The second term is finite: this is an hypothesis of the theorem.

For the last one, by Minkowski's inequality for the conditional expectancy and a fixed index $i$, since there are only finitely many terms in the sum, we obtain
\begin{multline}
\bbE \left[\left.\left( \sum\limits_{j \neq i : i \in \Gamma_j} \mathcal R_j^\alpha d_n^j \left(\rho_n^j,\eta_n^j\right)\right)^m \right\vert \mathcal R,\mathcal S,\mathcal L \right] \\
\begin{aligned}
&\leq \left( \sum\limits_{j \neq i : i \in \Gamma_j} \mathcal R_j^\alpha \bbE \left[\left. d_n^j \left(\rho_n^j,\eta_n^j\right)^m \right\vert \mathcal R,\mathcal S,\mathcal L \right]^{\frac{1}{m}} \right)^m\\
&\leq \left( \sum\limits_{j \neq i : i \in \Gamma_j} \mathcal R_j^\alpha \bbE \left[ Y^{m} \right]^{\frac{1}{m}} \right)^m\\
&\leq \left( \sum\limits_{j \neq i : i \in \Gamma_j} \mathcal R_j^\alpha \Vert Y \Vert_{L^m} \right)^m\\
&= \left( \sum\limits_{j \neq i : i \in \Gamma_j} \mathcal R_j^\alpha\right)^m \Vert Y \Vert_{L^m},
\end{aligned}
\end{multline}
where $\Vert Y \Vert_{L^m}$ stands for $\bbE \left[Y^m\right]^{\frac{1}{m}}$ which is finite from the first part of the theorem. Then we can take the supremum over all $i \in \Gamma $ and the mean, which is a finite from Hypothesis \eqref{hyphaut}.
\end{proof}

\subsection{Convergence of iterated sequences toward the fixed-point}
\label{sec:preuveCV}
\begin{llem} \label{lemiteres}
Under Hypotheses \eqref{hypLnonnul} and \eqref{condsufeqnY} with $p=1$, let $\delta$ be a probability distribution on $\kgp$ satisfying the hypothesis of Theorem \ref{thmcviteres}. For all $m \geq 0$ let $D_{\delta,1}^{(m)}$ be a random variable with distribution $\nu_1^{\phi^m (\delta)}$. Then $D_{\delta,1}^{(m)} \rightarrow D_1$ in distribution, and $\bbE [D_{\delta,1}^{(m)}] \rightarrow \bbE[D_1]$ as $m$ tends to infinity.
\end{llem}

\begin{proof}
Note that the matrices considered here have size 2 and are symmetric with $0$ on the diagonal, hence we identify them with their non diagonal coefficient in the remaining of the proof.

Since we work with nonnegative random variables now, the convergence in distribution is deduced from \cite[Lemma 4.5]{jelenkovic2012implicit} as for the converence of the mean. We have for all $m \geq 1$:
\begin{equation*}
D_{\delta,1}^{(m)}= Y_m + \sum_{\theta \in \Theta_m} C_\theta D_{\delta,1}^{(0,\theta)}
\end{equation*}
where $Y_m$ and $C_\theta$ are taken from Section \ref{sectionminimalsolu} and the $(D_{\delta,1}^{(0,\theta)})_{\theta}$ are iid copies of $D_{\delta,1}^{(0)}$ with distribution $\nu_1^\delta$, independent from the other parameters. $Y_m \rightarrow Y$ almost surely from Proposition \ref{uniquenessYL1}, hence we just need a convergence of the second term towards $0$ thanks to Slutsky's lemma. For all $\varepsilon>0$
\begin{equation*}
\begin{split}
\mathbb P \left(\sum_{\theta \in \Theta_m} C_\theta D_{\delta,1}^{(0,\theta)}> \varepsilon\right)& \leq \varepsilon^{-1} \bbE \left[\sum_{\theta \in \Theta_m} C_\theta D_{\delta,1}^{(0,\theta)}\right] \\
& \leq \varepsilon^{-1} \bbE \left[ \sum_{\theta \in \Theta_m} C_\theta \right] \bbE \left[ D_{\delta,1}^{(0)} \right]
\end{split}
\end{equation*}
where from Hypothesis \eqref{condsufeqnY} with $p=1$, $\bbE [ \Sigma_{\theta \in \Theta_m} C_\theta ]<\varrho^m$ with the same argument as in Equation \eqref{rhom}. The right-hand term then converges to $0$ as $m$ tends to infinity.
\end{proof}

\begin{proof}[Proof of Theorem \ref{thmcviteres}]
for integers $m,n \geq 1$, similarly to Lemma \ref{lemiteres}, we consider $D_{\delta,n}^{(m)}$ a random variable with distribution $\nu_n^{\phi^m(\delta)}$. The weak convergence of the measures in $\kgp$ is deduced from the convergence in distribution and in mean for all $n$ of
\begin{equation*}
D_{\delta,n}^{(m)} \rightarrow D_n
\end{equation*}
as $m$ tends to infinity, with Proposition \ref{critereCVmesures}.

Lemma \ref{lemiteres} corresponds to the case $n=1$ that we have to prove. For the induction on $n$, we refer to the proof of \cite[Theorem 2 (i)]{broutin2016self}: it relies on a contraction argument, where the equivalent of their $\bbE[U_n]$ in our setup is bounded by $\mathbb E \left[ \sum_{i \in \Gamma} \mathbf 1_{J \in \Gamma_i} \mathcal R_i^\alpha \right]<1$ from Hypothesis \eqref{condsufeqnY} with $p=1$.
\end{proof}

\section{Proof of the geometric properties of the fixpoint}
\label{sectiongeo}

\subsection{Upper bound on the upper Minkowski dimension}

In order to bound almost surely from above the upper Minkowski dimension of our metric space $\mathcal X$ fixed-point in distribution of Equation \eqref{eqnppl}, we need to count the cardinality of a covering of the space by balls of the same radius.

Intuitively, since $\mathcal X$ satisfies a recursion via Equation \eqref{eqnppl}, a good direction is to develop $\mathcal X$ for one or several levels of iteration. However, because $\Gamma$ can be infinite, one cannot hope to get a general bound on the upper Minkowski dimension by a naive inequality like
\begin{equation} \label{naiveminkowski}
N_{\mathcal X} (\varepsilon) \leq \sum_{i \in \Gamma} N_{\mathcal R_i^\alpha \mathcal X_i} (\varepsilon),
\end{equation}
because the sum is indexed by $\Gamma$ possibly infinite and each term is $1$ or greater. Instead of the naive bound \eqref{naiveminkowski}, we bound the cardinality of a covering of $\mathcal X$ by the coverings of the biggest spaces after rescaling together with a covering of the structure on which it is glued. The latter is the random metric space $\mathcal B$ on which we made several hypothesis concerning the geometry in Section \ref{sec:hyp}, and which can be seen as an independent copy of $\mathcal X_1^\emptyset$. A covering of this metric space should contain the little $\mathcal R_i^\alpha \mathcal X^i$ spaces. Under the right hypothesis, we should then have a relation close to
\begin{equation*}
N_{\mathcal X} (\varepsilon) \leq \sum_{i \in \Gamma} N_{\mathcal R_i^\alpha \mathcal X^i} (\varepsilon) \mathbf 1_{ \{\mathcal R_i^\alpha \mathcal X^i \mbox{ is "big"} \}} + N_{\mathcal B}(\varepsilon).
\end{equation*}
What is left to decide is the threshold at with we decide either to count the cardinality of the covering, or consider that it is inside the covering of $\mathcal B$.

We consider a "Cut-line" or a "Stopping-line" for the discrete fragmentation process described by the $\mathcal R_\theta$-s for $\varepsilon>0$ as follows:
\begin{equation}\label{cutspaceceps}
\mathcal C_\varepsilon := \left\{ \theta = i_{1} i_{2} \ldots i_n \in \Theta :  \mathcal R_{i_1} \geq \varepsilon , \ldots , \mathcal R_{i_1 \ldots i_{n-1}} \geq \varepsilon ,\mathcal R_\theta < \varepsilon \right\}.
\end{equation}

On this cut-line, $\mathcal X$ can be decomposed into subspaces $\mathcal X^\theta$ for $\theta \in \mathcal C_\varepsilon$, glued together onto what was added on the iterations of order $\theta$ for all indexes before $\mathcal C_\varepsilon$. We denote $\mathcal C_\varepsilon ^\uparrow := \left\{ \theta \in \Theta : \mathcal R_\theta \geq \varepsilon \right\}$ the set of the indexes above the cutline.

In order to obtain subspaces which distances are of order $\varepsilon$, we have to consider the cut-space $\cepsalpha$ to match our actual scaling. At this point, we cover the entire space $\mathcal X$ by covering the biggest spaces of $\cepsalpha$ by a finite bounded number of balls of radius $\varepsilon$, and also by covering the copies of $\mathcal B$ added and rescaled at all the indexes in $\cepsalphaf$ by a number of balls depending on $\dup$ which comes from Hypothesis \eqref{hypdimms}. If we chose well the threshold, the sufficiently small spaces of $\epsalpha$ will be included in balls from the covering of the copies of $\mathcal B$ above the cutline.

\begin{proof}[Proof of Theorem \ref{thmborndimms}]
As indicated, for a fixed $\varepsilon>0$ and the corresponding cutline $\cepsalpha$, we consider separately a covering of the big subpsaces on the cutline and one of the previously construced rescaled copies of $\mathcal B$, above the cutline. Several cases will be studied but the covering we consider proves to be sharp enough for the required bound on $\dimms \left(\mathcal X\right)$.
 
\paragraph{Big subspaces} We first bound the number of balls of radius $\varepsilon$ necessary to cover the biggest subspaces that appear on the cutline. By big subspaces, we mean the ones not included in the ball centered in the root $\rho_\theta$ of radius $\varepsilon/2$. The number of such subspaces $\nbigeps$ is described easily as
\begin{equation}\label{cardbigsubspaces}
\nbigeps := \sum\limits_{\theta \in \cepsalpha} \mathbf 1_{\mathcal R_\theta^\alpha h(\mathcal X^\theta)>\frac{\varepsilon}{2}}.
\end{equation}
By Markov's inequality, we can bound the conditional expectation of this random variable, conditionnaly on the cutline $\cepsalpha$ for a $p>0$:
\begin{equation*}
\bbE \left[ \nbigeps \vert \cepsalpha \right] \leq \sum\limits_{\theta \in \cepsalpha} \mathbb E \left[ \mathcal R_\theta^{\alpha p} \vert \cepsalpha \right] \mathbb E \left[ h(\mathcal X^\theta)^p\right] 2^p \varepsilon^{-p}.
\end{equation*}
We used a generalized Markov property, which ensures that $h(\mathcal X^\theta)$ is independent of $\cepsalpha$ for every $\theta$ and the $(h(\mathcal X^\theta))_{\theta \in \cepsalpha}$ are iid. Now for any choice of $p$ such that $\alpha p>1$, we have
\begin{equation*}
\sum\limits_{\theta \in \cepsalpha} \mathbb E \left[ \mathcal R_\theta^{\alpha p} \vert \cepsalpha \right] \leq 1,
\end{equation*}
hence we obtain
\begin{equation*}\bbE \left[ \nbigeps \vert \cepsalpha \right] \leq K \varepsilon^{-1/\alpha}
\end{equation*}
for a constant $K$, since we supposed the finiteness of $p$-moment for $h(\mathcal X)$.

Now in order to cover each of these subspaces by $\varepsilon$-balls, we remark here that for any $\theta \in \cepsalpha$, by hypothesis, $\mathcal R_\theta<\varepsilon^{\frac{1}{\alpha}}$ almost surely. Hence a good bound on the mean of the cardinality of a covering of the big subspaces would be
\begin{equation}\label{plbnbigeps}
K \varepsilon^{-\frac{1}{\alpha}} \mathbb E \left[ N_{\mathcal X} (1) \right].
\end{equation}

\paragraph{Above the cutline}\label{para:abovecutline} We now cover evarything that was constructed before reaching the cutline, namely rescaled copies of $\mathcal B$ indexed by $\cepsalphaf$, by balls of radius $\frac{\varepsilon}{2}$. Let $N_\varepsilon^{\uparrow}$ be the cardinality of such a covering. Here we chose $\frac{\varepsilon}{2}$ over $\varepsilon$ in order to avoid problems when we take a global covering of $\mathcal X$: from the $\frac{\varepsilon}{2}$-covering of $\mathcal B$ we turn it into a $\varepsilon$-covering simply by increasing the radius of the balls already chosen. Doing so, the spaces such that $\mathcal R_\theta^\alpha h(\mathcal X^\theta)<\frac{\varepsilon}{2}$ are all entirely inside balls with radius $\varepsilon$ centered in $\mathcal B$.

In order to cover $\mathcal B$ by balls of radius $\varepsilon/2$, we need at most $\left(\frac{\varepsilon}{2}\right)^{-\dup}$ of such balls, times the diameter, which has finite. As for a subspace rescaled in distance by $\mathcal R_\theta^\alpha$, we need $\left(\mathcal R_\theta^{-\alpha} \frac{ \varepsilon}{2}\right)^{-\dup} h(\mathcal X^\theta)$ balls instead.

The term to be computed is then
\begin{equation}\label{boulesblocks}
\bbE \left[ N_\varepsilon^\uparrow \right] = \bbE \left[ \sum_{\cepsalphaf} \left( \frac{2 \mathcal R_\theta ^\alpha}{\varepsilon} \right)^{\dup} h(\mathcal X^\theta)\right]=2^{\dup} \varepsilon^{-\dup} \bbE\left[ h(\mathcal X^\theta) \right] \bbE \left[ \sum_{\cepsalphaf} \mathcal R_\theta ^{\alpha \dup} \right] ,
\end{equation}
and we focus on the last expectation in equation \eqref{boulesblocks}.

We first compute this expectation for $\epsalpha=2^{-m}$
\begin{equation}\label{smhausdorff}
S_m:= \bbE \left[ \sum_{\cepsalphaf} \mathcal R_\theta ^{\alpha \dup} \right] = \bbE \left[ \sum_{\theta \in \Theta} \sum_{k=0}^{m-1} \mathbf 1_{2^{-k-1} \leq \mathcal R_\theta < 2^{-k}} \mathcal R_\theta^{\alpha \dup} \right]=\sum_{k=0}^{m-1} \smk.
\end{equation}

Now for a fixed $k$, we develop $\smk$ in order to make a $\mathcal R_\theta$ to the power $1$ appear. This exponent $1$ on which we emphasize here plays the role of the so-called \emph{Malthusian exponent} of the fragmentation process which genealogy is described by the family $\left(\mathcal R_\theta \right)_{\theta \in \Theta}$, as for instance in \cite{MR2253162}. We also distinguish wether $\mathcal R_{\theta^-}$ is greater or lesser than $2^{-k}$ where we recall that $\theta^-$ is the word such that there is $i \in \Gamma$ such that $\theta^- i = \theta$.
\begin{equation}\label{smk}
\begin{split}
\smk &= \bbE \left[ \sum_{\theta \in \Theta}  \mathbf 1_{2^{-k-1} \leq \mathcal R_\theta < 2^{-k}} \mathcal R_\theta \mathcal R_\theta^{\alpha \dup-1} \right] \\
&= \bbE \left[ \sum_{\theta \in \Theta}  \mathbf 1_{2^{-k-1} \leq \mathcal R_\theta < 2^{-k}} \mathcal R_\theta \mathcal R_\theta^{\alpha \dup-1} \left(\mathbf 1_{\mathcal R_{\theta^-} < 2^{-k}} + \mathbf 1_{\mathcal R_{\theta^-} \geq 2^{-k}} \right) \right].
\end{split}
\end{equation}

An index $\theta$ such that $\mathcal R_{\theta^-}\geq 2^{-k}$ and $\mathcal R_\theta<2^{-k}$ is by definition in $\mathcal C_{2^{-k}}$. If $\mathcal R_\theta<2^{-k}$ but $\mathcal R_{\theta^-}< 2^{-k}$, there is a unique couple $(\varphi,\psi) \in \Theta^2$ such that $\theta=\varphi \psi$ and $\varphi \in \mathcal C_{2^{-k}}$. We may then reindex the sum as follows:
\begin{equation}\label{smknlleforme}
\smk=\bbE \left[ \sum_{ \theta \in \mathcal C_{2^{-k}}} \mathbf 1_{\mathcal R_\theta \geq 2^{-k-1}} \mathcal R_\theta \mathcal R_\theta^{\alpha \dup-1} M_\theta \right],
\end{equation}
where
\begin{equation*}
 M_\theta:=\sum_{\varphi \in \Theta} ( \mathcal R_\varphi^{(\theta)})^{\alpha \dup} \mathbf 1_{\{ \mathcal R_\theta \mathcal R_\varphi^{(\theta)} \geq 2^{-k-1} \}}
\end{equation*}
which is always greater than $1$.

Since $\sum_{i \in \Gamma} \mathcal R_i=1$ almost surely, at most one of the $\mathcal R_i$s is greater than $\frac{1}{2}$. We denote by $I$ a random index such that for all $i \in \Gamma$, $\mathcal R_I \geq \mathcal R_i$ almost surely. This way $\mathcal R_I$ is the only $\mathcal R_i$ that can be greater than $\frac{1}{2}$. With this remark, it is clear that the indices of the $\mathcal R_\varphi$ appearing in the expression of $M_\theta$ are all some $\prod_{j=0}^m \mathcal R^{\theta I_1 \ldots I_{j-1}}_{I_j}$ for $m \in \mathbb N$.

Three cases appear from now on, depending on the sign of $\beta:=\alpha \dup-1$, leading to different bounds for $\dimms\left(\mathcal X \right)$.

\paragraph{a)} If $\beta>0$ which means that $\frac{1}{\alpha}<\dup$ we can bound $2^{-(k+1)\beta} \leq \mathcal R_\theta^\beta < 2^{-k\beta}$ almost surely if $2^{-k-1} \leq \mathcal R_\theta < 2^{-k}$.
Hence we can bound every $\smk$ from Equation \eqref{smknlleforme} as follows:
\begin{equation*}
\smk \leq \bbE \left[ \sum_{\theta \in \mathcal C_{2^{-k}}} \mathbf 1_{2^{-k-1} \leq \mathcal R_\theta < 2^{-k}} \mathcal R_\theta  M_\theta \right]2^{-k\beta}
\end{equation*}
 We have almost surely $M_\theta \leq \max \{ m \in \mathbb N : \prod_{j=0}^m \mathcal R^{\theta I_1 \ldots I_{j-1}}_{I_j} \geq \frac{1}{2} \}$. Indeed, if $\prod_{j=0}^m \mathcal R^{\theta I_1 \ldots I_{j-1}}_{I_j} < \frac{1}{2}$, then necessarily the corresponding $\mathcal R_\varphi$ that appears in the sum $\smk$ is not in the range $(2^{-k-1},2^{-k})$. Each $\mathcal R^{\theta I_1 \ldots I_{j-1}}_{I_j}$ is then bounded by $1$. Note that this bound is independent of the stopping line $\mathcal C_{2^{-k}}$, and that it does not depend on $k$ by the generalized Markov property.
\begin{equation}\label{smklastupper}
\smk \leq \bbE \left[ \sum_{\theta \in \mathcal C_{2^{-k}}} \mathbf 1_{2^{-k-1} \leq \mathcal R_\theta < 2^{-k}} \mathcal R_\theta  \right] \bbE \left[ M_\emptyset \right] 2^{-k\beta},
\end{equation}
from the remark that $M_\theta$ is independent of the stopping line $\mathcal C_{2^{-k}}$. In order to understand the last term, we remark that it is the same as in \cite[Section 1.4.4]{MR2253162}. 
\begin{equation*}
\sum_{\theta \in \mathcal C_{2^{-k}}} \mathbf 1_{2^{-k-1} \leq \mathcal R_\theta < 2^{-k}} \mathcal R_\theta =  \nu_{2^{-k}} \left( \left[ \frac{1}{2},1 \right] \right)
\end{equation*}
with $\nu_\eta$ corresponding to the $\varphi_\eta$ in the notations of \cite{MR2253162}. Note that the results of Bertoin from \cite{MR2253162} apply in our case since it only depends on the genealogical tree introduced earlier in the monograph, corresponding to our $\Theta$. The following theorem in \cite{MR2253162} states that $\nu_{\eta}$ is a random measure on $[0,1]$ that converges in probability as $\eta$ tends to $0$ towards a deterministic measure $\nu$ which expression is given by:
\begin{equation*}
\nu(da):= \bbP \left(  \mathcal R_* <a \right) \frac{da}{a \bbE \left[ -\log \mathcal R_*\right]},
\end{equation*}
where $\mathcal R_*$ is defined in Section \ref{sec:resultdimf} and represents a random $\mathcal R_i$ chosen with a probability biased by its size.

Hence $\left( \bbE \left[ \nu_{2^{-k}} \left( \left[ \frac{1}{2},1 \right] \right) \right]\right)_k$ converges towards $\nu\left( \left[ \frac{1}{2},1 \right] \right)$ and in particular the sequence is upper bounded.

We obtain the ultimate bound on $\smk=O \left( 2^{-(k+1) \vert \beta \vert} \right)$ which leads to 
\begin{equation*}
S_m = O \left( \sum_0^{m-1}  2^{-(k+1) \vert \beta \vert} \right)= O(1).
\end{equation*}

As a partial conclusion in this case $\beta>0$, we have $\bbE \left[ N_\varepsilon^\uparrow \right] =O(\varepsilon^{-\dup})$.

\paragraph{b)} If $\beta<0$, which means $\frac{1}{\alpha}>\dup$, we can bound $2^{-(k+1)\beta} \geq \mathcal R_\theta^\beta > 2^{-k\beta}$ almost surely if $2^{-k-1} \leq \mathcal R_\theta < 2^{-k}$:
\begin{equation}
\sum\limits_{k=0}^{m-1} \bbE \left[\nu_{2^{-k}}([\frac{1}{2},1]) \right] 2^{-k \beta} \leq S_m \leq \sum\limits_{k=0}^{m-1} \bbE \left[\nu_{2^{-k}}([\frac{1}{2},1]) \right] 2^{-(k+1) \beta} \bbE \left[M_\emptyset\right].
\end{equation}
On both sides, since $(\nu_{2^{-k}})_k$ converges in probability towards $\nu$ again from \cite{MR2253162}, $(\bbE \left[\nu_{2^{-k}}([\frac{1}{2},1]) \right])_k$ converges towards a positive constant $\ell:=\nu \left(\left[\frac{1}{2},1\right]\right)$. For any $\eta>0$ such that $\eta<\ell/2$, we can find $q$ sufficiently large after which each term of the sequence is separated from $\ell$ at most by a distance $\eta$. From there we only look at the terms of the sequence after that $q$ and condense the constant terms on both sides by $K$ and $K'$:
\begin{equation*}
K + \sum\limits_{k=q}^m \frac{\ell}{2} 2^{-k \beta} \leq S_m \leq K' + \sum\limits_{k=q}^m \frac{3 \ell}{2} 2^{-(k+1) \beta} \bbE\left[M_\emptyset\right]
\end{equation*}
Since $\beta$ is negative, each sum is equivalent to a constant times $2^{-m \beta}$. Hence $S_m$ is of order $(\varepsilon^{\frac{1}{\alpha}})^{ -\beta}=\varepsilon ^{\dup-\frac{1}{\alpha}}$. We can finally precise the power law behaviour:
\begin{equation}\label{betaneg}
\bbE \left[ N^\uparrow_\varepsilon \right] \mbox{ is of order } \varepsilon^{-\frac{1}{\alpha}}.
\end{equation}

\paragraph{c)} The last case to notice is the intermediary when $\frac{1}{\alpha}=\dup$ which leads to $\mathcal R_\theta^\beta= \mathcal R_\theta^0=1$ almost surely. The proof leading to \eqref{betaneg} can be adapted to give an equivalent of order $\varepsilon^{-\dup} \log \varepsilon$, where the $\log$ term vanishes when we only focus on the power-law behaviour.

\paragraph{Conclusion} We construct a covering of the entire space $\mathcal X$ as follows. The big subspaces of $\cepsalpha$ are covered as described previously by balls of radius $\varepsilon$, which constitut $N^b_\varepsilon$ balls. The little subsets of the cutline which have height lesser than $\varepsilon/2$ fall inside the balls constituing the covering of what lies above the cutline described in Paragraph \ref{para:abovecutline}, on which we changed the radius from $\varepsilon/2$ to $\varepsilon$. Doing so we obtain a covering of $\mathcal X$ which cardinal is $N_\varepsilon:= N_\varepsilon^{\mathrm b} + N_\varepsilon^\uparrow$.

This leads in both case \textbf{a)} and \textbf{b)} that $\bbE [ N_\varepsilon ] = O \left( \varepsilon^{-\max \left(\frac{1}{\alpha} , \dup \right)} \right)$, at least for $\varepsilon=2^{-\alpha m}$. The intermediate situation \textbf{c)} gives the same power-law behaviour.

Using Markov's inequality and the Borel-Cantelli lemma, for any $\varrho> \max ( \dup , \frac{1}{\alpha})$, it entails that $N^b_{2^{-\alpha m}}2^{-\varrho m} \rightarrow 0$. From the monotony of $N_x^b$ in $x$, we deduce that $N_x^b x^\varrho \rightarrow 0$ almost surely as $x$ tends to $0$.
\end{proof}

\subsection{Lower bound on the Hausdorff dimension}
\label{sec:proofdimh}

In order to bound from below the Hausdorff dimension of a metric space $(\mathcal X,d)$, we construct a Borel probability measure on $\overline \mu$ on $\mathcal X$ in the next Proposition \ref{propmubarre} and apply the mass distribution principle recalled in \ref{frostman}. Intuitively, this measure is obtained by scaling at each step by the corresponding $\mathcal R$ instead of the $\mathcal S$ that we have used before. As in \cite{broutin2016self}, this measure is the same as the measure we constructed recursively on $\mathcal X$ in the case $\mathcal R = \mathcal S$.
\begin{lpro}\label{propmubarre}
Almost surely, for all $\theta$, there exist a (random) probability measure $\overline{\mu^\theta}$ on $\mathcal X$ such that $\overline{\mathcal X} := \left( X^\theta, d^\theta, \overline{\mu^\theta}, \rho_m^\theta \right)$ is a random rooted measured metric space and for every $\sigma \in \Theta$:
\begin{equation}
\overline{\mu^\theta} \left(\mathcal X^\theta_\sigma \right) = \frac{\mathcal R_{\theta \sigma}}{\mathcal R_\theta}.
\end{equation}
\end{lpro}

\begin{proof}
We first define for all $n \in \mathbb N$ the measure $\overline{\mu^\theta_n}$ on $X^\theta$ such that for all $\sigma \in \Theta_n$ we have
\begin{equation}
\overline{\mu^\theta_n} \left(X_\sigma^\theta \right):=\frac{\mathcal R_{\theta \sigma}}{\mathcal R_\theta}.
\end{equation}
By doing this, we have defined $\overline{\mu^\theta_n} \left(X_\sigma^\theta \right)$ for all $\sigma$ such that $\vert \sigma \vert \leq n$. We make the noteworthy remark that this construction is consistent in the sense that for all $m \geq n$ and $\sigma \in \Theta_n$, we have
\begin{equation}
\overline{\mu^\theta_m} \left(X_\sigma^\theta \right)=\overline{\mu^\theta_n} \left(X_\sigma^\theta \right).
\end{equation}
The measure $\overline{\mu^\theta} $ is constructed as the almost sure limit  of the sequence of random measures $(\overline{\mu^\theta_n})_n$.
\end{proof}

Again we will iterate the application $\phi$ on the space $\mathcal X$ so that it is seen as a glueing of several rescaled copies of itself onto rescaled copies of $\mathcal B$. The $\frac{1}{\alpha}$ is an asymptotic bound on the almost sure Hausdorff dimension of $\mathcal X$, in the sense that it appears as we iterate $\phi$ for a great number of times. As for the upper Minkowski dimension, we need to take into account all the lower levels of iteration.

Indeed, if we consider the first iteration of $\phi$, it contains the metric space $\mathcal B$. Hence a first bound on the almost sure Hausdorff dimension of $\mathcal X$ would be
\begin{equation}
\dimh(\mathcal X) \geq \dlow.
\end{equation}

\begin{llem}\label{lemmedimH}
Assume that $\mathbb E \left[ \mathcal R_*^{-\delta} \right] < +\infty$ for a $\delta >0$. Let $\zeta$ be drawn on $\mathcal X$ according to measure $\mu$ and $\overline{\zeta}$ be with distribution $\overline{\mu}$ on the same $\mathcal X$ such that $\zb$ is independent from $\zeta$. Then we have:
\begin{equation*}
\begin{split}
&\mathbb P \left (d(\rho,\zb)<r\right) = O (r^\kappa) \\
& \mathbb P \left (d(\zeta,\zb)<r\right) = O (r^\kappa)
\end{split}
\end{equation*}
when $r$ is close to $0$, for a $\kappa>0$.
\end{llem}

\begin{proof}
For every $n \in \mathbb N$, we denote by $i_n$ the random variable in $\Gamma$ such that after iterating $n$ times, we can locate $\zb \in \mathcal X^{i_1 \dots i_n}$ (again we omit to precise the canonical projections). It is almost surely unique, and allows us to trace the genealogy of the random point $\zb$.

\paragraph{Bound on $d(\rho,\zb)$} Now consider the sequence of random variables
\begin{equation}
\left(\mathcal L^{(i_1 \dots i_{n-1})} \left( 1 , i_n \right)\right)_{n \in \mathbb N}
\end{equation}
which are iid random variables with the same distribution as $\mathcal L \left( 1 , I \right)$ where we define $I$ such that $\mathcal P \left(I=k \vert \mathcal R , \mathcal S ,\mathcal L \right) = \mathcal R_k$ for all $k \in \Gamma$, as in the definition of $\mathcal R_*$ in Section \ref{sec:resultdimf}. From Hypothesis \eqref{hypLnonnul}, we have $\mathbb E \left[\mathcal L \left( 1 , I \right) \right]>0$.

Let $\eta:=\mathbb E \left[\min\left\{ \mathcal L \left( 1 , I\right),1\right\}  \right] \in (0,+\infty)$. We consider the random time
\begin{equation*}
T:=\min\left\{ n \in \mathbb N : \mathcal L ^{(i_1 \dots i_{n-1})} \left( 1 , i_n \right) >\eta \right\},
\end{equation*}
its distribution is geometric with parameter $p=\mathbb P \left(\min\left\{ \mathcal L \left( 1 , I\right),1\right\} >\eta\right) \in (0,1)$.

For all $n\in \mathbb N$, we have
\begin{equation*}
d(\rho,\zb) \geq \mathcal L ^{(i_1 \dots i_{n-1})} \left( 1 , i_n \right) \left(\prod_{k=1}^{n} \mathcal R_{i_k}^{(i_1 \ldots i_{k-1})}\right)^\alpha.
\end{equation*}
This bound is obtained by onyl considering the portion of the geodesic between $\rho$ and $\zb$ inside the subspace $\mathcal X^{i_1 \dots i_{n}}$. Hence because $T$ is a stopping time for the filtration generated by $(\mathcal R^{(\theta)} , \mathcal S^{(\theta)} , \mathcal L^{(\theta)})_{\theta \in \Theta_n}$, the same inequality holds at the random rank $T$ where we get
\begin{equation*}
d(\rho,\zb) \geq \eta \left(\prod_{k=1}^{T} \mathcal R_{i_k}^{(i_1 \ldots i_{k-1})}\right)^\alpha
\end{equation*}
Then we obtain for a parameter $\gamma>0$
\begin{equation}\label{bounddrzbE}
\mathbb P \left(d(\rho,\zb) \leq r \right) \leq \mathbb P \left( \prod_{k=1}^T\left( \mathcal R_{*,k} \right)^{-\alpha} \geq \frac{\eta}{r} \right) \leq \mathbb E \left[ \prod_{k=1}^T\left(\mathcal R_{*,k} \right)^{-\alpha \gamma} \right] \left(\frac{r}{\eta} \right)^\gamma,
\end{equation}
where $\left(\mathcal R_{*,k} \right)_k$ is a sequence of iid copies of $\mathcal R_{*} $. Furthermore, by separating on the events $\{T=n\}$, and by independence of the $\mathcal R_{*,k}$:
\begin{equation*}
\mathbb E \left[ \prod_{k=1}^T\left( \mathcal R_{*,k} \right)^{-\alpha \gamma} \right]=\mathbb E \left[ \exp \left(\alpha \gamma \sum_{k=1}^T -\log \mathcal R_{*,k}  \right) \right]= \sum_{n=1}^\infty \mathbb P (T=n) \bbE \left[ \mathcal R_*^{-\alpha \gamma}\right]^n,
\end{equation*}
which is finite if we choose $\gamma$ such that $\alpha \gamma \leq \delta$.

\paragraph{Bound on $d(\zeta,\zb)$}
Again we study the space $\mathcal X$ after $n$ iterations, and focus on the portion of the geodesic between $\zeta$ and $\zb$ inside the subspace in which falls $\zb$ namely $\mathcal X^{i_1 \dots i_{n}}$. The distance $Z_n$ it covers, if we omit rescaling in a first look, is either an independent copy of $d(\zeta,\zb)$, or an independent copy of $d(\rho,\zb)$. To be more precise, we refer to the tools in the proof of \cite[Lemma 16]{broutin2016self}: a copy of $d(\zeta,\zb)$ appears whenever the geodesic exits $\mathcal X^{i_1 \dots i_{n}}$ via the root $\rho^{i_1 \dots i_{n}}$ of this subset, whereas we have $d(\zeta,\zb)$ when the exit point of $\mathcal X^{i_1 \dots i_{n}}$ is a branchpoint $\eta_i$ at a level of iteration lesser than $n$ or when $\zeta \in \mathcal X^{i_1 \dots i_{n}}$.

After these considerations, we want to follow the proof for the bound on $d(\rho,\zb)$ and consider the sequence $(L_n)_n$ of random variables such that for all $n$:
\begin{equation*}
L_n := \min \left\{\mathcal L^{(i_1 \dots i_{n-1})} \left( J_n , i_n \right), \mathcal L^{(i_1 \dots i_{n-1})} \left( 1 , i_n \right)\right\},
\end{equation*}
with $J_n$ such that $\mathbb P \left(J_n =k \vert \mathcal R^{(i_1 \dots i_{n-1})},\mathcal S^{(i_1 \dots i_{n-1})},\mathcal L^{(i_1 \dots i_{n-1}}\right)= \mathcal S^{(i_1 \dots i_{n-1})}_k$. This sequence is iid: $L_n$ only depends on $(\mathcal R^{(i_1 \dots i_{n-1})},\mathcal S^{(i_1 \dots i_{n-1})},\mathcal L^{(i_1 \dots i_{n-1})})$.

Again $\bbE [\min\{L_n,1\}] \in (0,+\infty)$, and the remaining of the proof is the same as in the previous case.
\end{proof}

\begin{proof}[Proof of Theorem \ref{thmborndimh}]
The bound $\dimh \left( \mathcal X \right) \geq \dlow$ is immediate: by iterating once $\mathcal X$ contains a copy of $\mathcal B$. It remains to prove that $\dimh\left(\mathcal X\right) \geq \frac{1}{\alpha}$. To do so, We fix a $\gamma < \frac{1}{\alpha}$ with which we want to apply the mass distribution principle stated in Equation \eqref{frostman}. The measure on which we apply it will be the $\overline{\mu}$ we constructed on $\mathcal X$ in Proposition \ref{propmubarre}.

The remaining of the proof can be directly adapted from \cite[Theorem 5]{broutin2016self}. We refer to their proof for the details and only focus here on the notable differences.

Indeed, we have to bound $\mathbb P \left( \overline{\mu} (B_r(\zb))>r^\gamma\right)$ by a power of $r>0$ in order to apply Borel-Cantelli Lemma. The main difference between the two proofs lies on the fact that we look at balls centered in points drawn according to $\overline{\mu}$ rather than $\mu$. This only difference will have the same consequences since our Lemma \ref{lemmedimH} is similar to \cite[Lemma 18]{broutin2016self}, without the result concerning $d(\rho,\zeta)$ which we do not need when we sample according to $\overline{\mu}$.
\end{proof}

\section{Proof of the applications}
\label{sec:appliproof}

\subsection{Stable looptrees}
\label{sectionlooptrees}

\subsubsection{Looptree encoding of a càdlàg excursion with positive jumps and $\beta$-stable looptrees}
\label{sectioncodage}

Let $f$ be such a right-continuous with left limits (càdlàg) excursion taking values in $\mathbb R$, with positive jumps. We denote by $\Delta_t:=f(t)-f(t-) \geq 0$ the height of the jump at time $t \in (0,1]$, and set $\Delta_0:=0$.

We define a partial order on $[0,1]$ denoted by $\preccurlyeq$ by setting:
\begin{equation*}
s \preccurlyeq t \mbox{ if } s \leq t \mbox{ and } f(s-) \leq \inf_{[s,t]} f.
\end{equation*}
For any $s,t \in [0,1]^2$, let $s \wedge t$ be the most recent common ancestor of $s$ and $t$ with respect to the relation $\preccurlyeq$. We also note for any $s,t$ such that $s \preccurlyeq t$:
\begin{equation*}
x_s^t := \inf_{[s,t]} f - f(s-) \in [0,\Delta_s],
\end{equation*}
it will further describe the relative position of $t$ on the jump at level $s$. For every $t \in [0,1]$, we equip $[0,\Delta_t]$ with the pseudo-distance $\delta_t$ defined by
\begin{equation*}
\delta_t (a,b) := \min \{ \vert a-b \vert , \Delta_t - \vert a-b \vert \}, \mbox{ for } a,b \in [O,\Delta_t].
\end{equation*}
If we write $s \prec t$ if $0 \leq s \leq t \leq 1$ with $s\preccurlyeq t$ and $s \neq t$, we set now for $s \preccurlyeq t$:
\begin{equation*}
d_0 (s,t) := \sum\limits_{s \prec r \preccurlyeq t} \delta_r(0,x_r^t),
\end{equation*}
and with that we can define the distance between any two points $s,t \in [0,1]$:
\begin{equation}\label{pseudodloop}
d(s,t):= \delta_{s \wedge t} (x^s_{s \wedge t},x^t_{s \wedge t}) + d_0( s \wedge t,t) = d_0 (s \wedge t,s).
\end{equation}

Equation \eqref{pseudodloop} defines a pseudo-distance on $[0,1]$. And we construct the looptree associated to $f$ by taking the quotient of $[0,1]$ by this pseudo-distance $d$. Intuitively, the looptree associated to a function $f$ consists in loops glued together in a tree shape, each loop being associated with a jump of the process.

\paragraph{The $\beta$-stable looptree} From now on, let $\beta\in (1,2)$ be a fixed parameter. We will consider a \emph{$\beta-$stable Lévy process} as presented in \cite{curien2014random} and we refer to \cite{MR1406564} for the proofs. All the processes we consider still have only positive jumps.

Let $\Xexc$ be the normalized excursion of this \emph{$\beta-$stable Lévy process} above its infimum, which has only positive jumps, as defined in \cite{MR1465814}. Notice that $\Xexc$ is almost surely a random càdlàg function on $[0,1]$ such that $\Xexc_0=\Xexc_1=0$ and $\Xexc_s>0$ for all $s \in (0,1)$.

We now apply the (deterministic) construction of the beginning of this section to the random càdlàg function $\Xexc$. The \emph{random $\beta$-stable looptree} is defined as the quotient metric space
\begin{equation}
	\mathscr L_\beta := \left([0,1] /\sim , d \right),
\end{equation}
where $\sim$ represents the equivalence relation $x \sim y$ if $d(x,y)=0$ with $d$ as in Equation \eqref{pseudodloop}. It is a fact that $\mathscr L_\beta$ is almost surely compact (see \cite{curien2014random} the discussion after Definition 2.3).

For more details on stable looptrees, we refer one more time to \cite[Section 2]{curien2014random}. Note that as it is defined in the original setting, $\mathscr L_\beta$ is a random variable on the set of Gromov-Hausdorff isometry classes of compact metric spaces. As we work with $\gp$-equivalence classes in this article, we add a probability measure and a root on the $\beta$-stable looptree. Following \cite[Section 3.3.2]{curien2014random}, we endow $\mathscr L_\beta$ with the probability measure $\mu$ obtained as the push-forward of the Lebesgue measure on $[0,1]$ by the canonical projection mapping $[0,1]$ to $\mathscr L_\beta$, and add a root $\rho$ as the image of $0$ by the same canonical projection.

From now on, $\mathscr L_\beta= \left([0,1]/\sim , d , \mu , \rho \right)$ is considered a rooted measured metric space, and more generally $\mathscr L_\beta$ will denote its $\gp$-equivalence class.

\subsubsection{The random $\beta$-stable looptree as the unique fixed-point of a distributional equation}

\begin{proof}[Proof of Theorem \ref{thmlooptreev2}]
Most of the computation of the distributions involved for this theorem are obtained from \cite[Theorem 3.8]{chee2018recursive} by Chee, Rembart and Winkel. In this article, the authors construct the $\beta$-stable tree $\mathcal T_\beta$ as the solution of a recursive distribution equation, glueing together countably many  iid copies of $\mathcal T_\beta$, rescaled by some random scaling parameters which distribution is made explicit. The scaling factors they obtain are linked directly to the genealogy of the stable tree which is induced by the $\beta$-stable excursion coding the tree and the partial order it induces on $[0,1]$ as recalled in Section \ref{sectioncodage}. The $\beta$-stable looptree coded from this same excursion process then shares the same splitting of mass: we obtain the scaling factors as in Equation \eqref{parameterslooptree}.

When we rescale looptrees, we multiply measures by $c$ and distances by $c^{1 / \beta}$.

The authors of \cite{chee2018recursive} were not interested in the size of the loop nor in the repartition of the substructures around it since it does not appear when we look at the stable tree. From \cite[Equation (1)]{MR2123249} and the linked remark of \cite{curien2014random}, one can recover the \emph{width of the hubs} of the stable tree with the formula
\begin{equation*}
L(u) = \lim\limits_{\varepsilon \rightarrow 0} \frac{1}{\varepsilon} \mu_{\mathcal T_\beta} \left\{v \in \mathcal T_\beta : d_{\mathcal T_\beta} \left(u,v \right) < \varepsilon \right\} \mbox{ a.s.}
\end{equation*}
for a node $u$ in $\mathcal T_{\beta}$. The width of nodes are in one to one correspondence with the jumps of the $\beta$-stable excursion coding the tree (Proposition 2 in \cite{MR2123249}), and we have $\Delta_s=L(u)$ where $u$ is the image of the point $s \in [0,1]$ in the tree. When we come back to the $\beta$-stable looptrees, each jump of height $\Delta_s$ of the excursion is sent to a circle with lenght $\Delta_s$ (again from \cite{curien2014random}, see the discussion at the end of Section 2). It remains to determine the law of the local time, or the \emph{width}, of the internal node around which rescaled stable trees are glued in the construction of \cite{chee2018recursive} which will be the length of the circle for our case of the $\beta$-stable looptree $\mathscr L_\beta$. From \cite{pitman2006combinatorial} we know that it will be a rescaled mutiple of the $1/\beta$-diversty of the sequence $(\mathcal R_i)_i$:
\begin{equation*}
\lim\limits_{i \rightarrow \infty} \Gamma \left(1- \frac{1}{\beta} \right) i \left(\mathcal R_i \right)^{\frac{1}{\beta}} = X_3^{\frac{1}{\beta}} \lim\limits_{i \rightarrow \infty} \Gamma \left(1- \frac{1}{\beta} \right) i \left(P_i^{\downarrow} \right)^{\frac{1}{\beta}},
\end{equation*}
from the expression for the scaling parameters $\mathcal R_i$ of the stable Looptree. 
Remark againt that it is distributed as $X^{\frac{1}{\beta}}$ times a random variable which distribution is $\mittagl \left( \frac{1}{\beta} , \frac{2}{\beta}-1 \right)$.

Lastly concerning the uniform position at which each sub-looptree is plugged onto the circle, it is deduced from the invariance by rerooting observed in \cite{curien2014random}.
\end{proof}

\begin{proof}[Proof of Corollary \ref{cor:ltpointfixe}]
This corollary is the application of Theorem \ref{thmppl} to the parameters of the stable looptree described above. In order to have a characherization of the $\beta$-stable looptree as a fixed-point, one needs to verify hypotheses \eqref{hypLnonnul} and \eqref{condsufeqnY} with $p=1$. The last one is straightforward since $\mathcal L(\emptyset,J)$ has a Mittag-Leffler distribution up to multiplication by bounded random variables.

From \cite{chee2018recursive}, in which the authors obtain a fixed-point equation for the $\beta$-stable tree in a setup falling out of our scope (there is no added distance at each step), Theorem 3.8, we can recover the same parameters $\mathcal R_i$ that are described in Section \ref{subsubs:lt}. The identity for the $\beta$-stable obtained by the authors gives an identity for the distance $Z$ between the root and a random point:
\begin{equation}
Z \loi \sum_i \mathbf 1_{J \in \Gamma_i} \mathcal R_i^{1-\frac{1}{\beta}}Z^{(i)},
\end{equation}
in which all the terms present are similar to what we encountered in Equation \eqref{eqndptunif} without a term in $\mathcal L$. We take the mean:
\begin{equation}
\bbE \left[\sum_i \mathbf 1_{J \in \Gamma_i} \mathcal R_i^{1-\frac{1}{\beta}}\right]=1.
\end{equation}
Note that the exponent $1-\frac{1}{\beta}$ is the scaling for the $\beta$-stable tree, which lies in $\left( 0, \frac{1}{2}\right)$.

Concerning the stable looptree, we recall that the parameters $\mathcal R_i$ are the same since they come directly from the stable process coding both the tree and the looptree, but we take them to the power $\frac{1}{\beta} \in \left(\frac{1}{2},1\right)$. By monotonicity of the function
\begin{equation}
x \mapsto \bbE \left[\sum_i \mathbf 1_{J \in \Gamma_i} \mathcal R_i^{x} \right],
\end{equation}
we conclude that $\bbE \left[\sum_i \mathbf 1_{J \in \Gamma_i} \mathcal R_i^{\frac{1}{\beta} } \right]<1$ as required.
\end{proof}

\subsubsection{Geometric results}
\begin{proof}[Proof of Corollary \ref{cordimLT}] The metric space $\mathcal B$ on which the iid copies are glued onto is isometric to a circle. It follows that Hypotheses \eqref{hypdimms} and \eqref{hypdimhblock} are satisfied with the constant $\dup=\dlow=1$. We may now apply both theorems relative to the bounds on fractal dimensions:
\paragraph{Upper bound on $\dimms(\mathscr L_\beta)$}
Hypothesis \eqref{hypdimms} is satisfied with $\dup=1$ which is the Minkowski dimension of a circle. Hypothese \eqref{hyphaut} holds as the term $\sup_i \mathcal L(\emptyset,i)$ has Mittag-Leffler distribution, which has finite moments of all orders. The last hypothesis needed to apply Theorem \ref{thmborndimms} is a consequence of Proposition \ref{propfiniteH}: the structural tree $\Gamma$ has height $2<+\infty$. This gives the first bound
$$\dimms(\mathscr L_\beta) \leq \beta=\max\{\beta,1\}$$

\paragraph{Lower bound on $\dimh(\mathscr L_\beta)$}

With the description of the parameters from Theorem \ref{thmlooptreev2}, for a $\delta>0$
\begin{equation}
\bbE \left[ \mathcal R_* ^{-\delta} \right] = \bbE \left[ X_0 ^{1-\delta}+X_1 ^{1-\delta} + X_2 ^{1-\delta} + X_3 ^{1-\delta} B^{-\delta} \right],
\end{equation}

where $B$ is a size-biased pick from $(P_i^\downarrow)$ whose distribution is $\lbeta(1-\alpha,3\alpha-1)$ (see Section \ref{dirpoidir}), and the $X_i$ are independent from $B$ with other $\lbeta$ distributions. From there a choice of $\delta< \inf\left\{1-\alpha , 2 \alpha,2-\alpha\right\}$ is always possible and gives the desired negative moment for $\mathcal R_*$. Hypothesis \eqref{hypdimhblock} is satisfied with $\dlow=1$, hence the expected lower bound by Theorem \ref{thmborndimh}.
\end{proof}

\subsection{Stable trees}

\subsubsection{The random $\beta$-stable tree as the unique fixed-point of a distributional equation}

\begin{proof}[Proof of Theorem \ref{thmarbrestable}]

The decomposition we present here is only a reformulation in our own terms of the spinal mass decomposition of the stable tree. Most of it is displayed in \cite[Corollary 10]{haas2009spinal}, to which we re-rooted the tree to get a size biased pick among the subspaces glued together (it is necessary to fit our setup).

The last arguments concerning the height between the root and a random point comes from \cite[Theorem 12]{addario2019inverting}.
\end{proof}

\begin{proof}[Proof of Corollary \ref{corstablet}]
Again we need to verify the hypothesis of Theorem \ref{thmppl}. The moment hypothesis  concerning $\mathcal L (\emptyset,J)$ is straightforward since this random variable is bounded almost surely by a constant times a Mittag-Leffler distribution, which has finite moments at all orders.

As for the second part of Hypothesis \eqref{condsufeqnY} with $p=1$, it can be deduced from the Tree itself without calculating $\bbE \left[\sum_i \mathbf 1_{J \in \Gamma_i} \mathcal R_i^\alpha\right]$. Indeed, on the $\beta$-stable tree, the spinal mass decomposition in which we reroot the tree by uniformly sampling a new root gives us a relation satisfied by the distance between this new root and a random point $Y$:
\begin{equation}\label{eq:corarbrestable}
Y \loi \sum_i \mathbf 1_{J \in \Gamma_i} \mathcal R_i^\alpha Y^{(i)} + \mathcal L(\emptyset,J).
\end{equation}
Since the $Y^{(i)}$s are random variables with distribution $\mittagl\left(1-\frac{1}{\beta},1-\frac{1}{\beta}\right)$ up to a constant, it remains when we take the mean in Equation \eqref{eq:corarbrestable} and simplify
\begin{equation*}
1 = \bbE \left[\sum_i \mathbf 1_{J \in \Gamma_i} \mathcal R_i^\alpha\right]+ \bbE\left[ \vert U_1-U_2 \vert \right],
\end{equation*}
which gives the required condition since the last term is positive.
\end{proof}

\subsubsection{Geometric results}

\begin{proof}[Proof of Corollary \ref{cordimT}]
The iid copies are glued at each step of the recursion onto a segment. Hence we have again $\dup=\dlow=1$. From there, Hypotheses \eqref{hyphaut}, \eqref{hypdimms} and \eqref{hypdimhblock} in the case of the $\beta$-stable follow from the same arguments as for the $\beta$-stable looptree: we refer to the proof of Corollary \ref{cordimLT}.

For $\delta>0$,
\begin{equation}
\bbE \left[ \mathcal R_* ^{-\delta} \right] = \bbE \left[ B_P ^{1-\delta} B_Q^{1-\delta} \right],
\end{equation}
with $B_P$ and $B_Q$ corresponding each to the size-biased picks from the two levels of independent Poisson-Dirichlet distributions. From the remark in Section\ref{dirpoidir}, they are independent random variables and have $\lbeta$ distributions. We can deduce that $\bbE \left[ \mathcal R_* ^{-\delta} \right]< \infty$ for $d < \min \left\{\frac{1}{\beta},1-\frac{1}{\beta}\right\}$, which is always greater than $0$. This is enough to provide the negative moment condition of Theorem \ref{thmborndimh}.
\end{proof}

\bibliographystyle{plain}
\bibliography{existuniqkgp.bbl}

\appendix
\section{Measurability of $\psi$}

\begin{llem} \label{lem:mes}
The application $\psi: (\kgp)^\Gamma \times \Sigma_\Gamma^2 \times D_\Gamma \rightarrow \mathcal M_1(\kgp)$  is measurable.
\end{llem}

\begin{proof}
Let us consider an indexation of $\Gamma$ on $\mathbb N$ respecting the genealogy of the tree. It means that for any integer indices $(i,j) \in \Gamma$, if $i$ is an ancestor of $j$, then $i\leq j$.

For $I \in \mathbb N$, consider the application $\psi^{(I)}$ defined on the same space as $\psi$ as follows:
\begin{equation*}
\psi^{(I)}((\mathfrak X_i)_{i \in \mathbb N},\mathbf r , \mathbf s,\ell))=\psi((\mathfrak X_0, \mathfrak X_1 , \ldots , \mathfrak X_I, \{ \rho_{I+1} \}, \{ \rho_{I+2} \} , \ldots) ,\mathbf r , \mathbf s,\ell)),
\end{equation*}
where the $\{ \rho_{k} \}$-s are singleton spaces. In other words, $\psi^{(I)}$ is the distribution of the $I+1$ first spaces $\mathfrak X_i$ glued together respecting the same rule as $\psi$, to which are added atoms of mass $s_i$ for all the indices $i>I$.

This application is continuous: it can be seen as a straightforward application of Lemma 30 in the appendix of \cite{broutin2016self} since we are concerned only with a finite number of spaces and hence a finite number of exit points $\eta_i$ to couple.

For the same fixed $Y:=\left((\mathfrak X_i)_{i \in \mathbb N},\mathbf r , \mathbf s,\ell \right)$, there is a natural way to couple random variables with respective distribution $\psi \left( Y \right)$ and $\psi^{(I)} \left( Y \right)$ by simply taking the output of the constructions with the same choice for each $\eta_i$ with $i\leq I$. Now conditionally
on these two random variables, we couple the two measures $\mu$ and $\mu^{(I)}$ of the random rooted measured metric spaces as follows: first sample $J$ with distribution $\mathbb P(J=k)=s_k$ and $\zeta_i$ with distribution $\mu_i$ on each $\mathfrak X_i$, then we define $\zeta:= \zeta_J$ which has distribution $\mu$ and $\zeta^{(I)}$ as $\zeta_J$ if $J\leq I$ and equal to $\rho_J$ if $J>I$.

Let $Z:= \mathfrak X \sqcup \cup_{j>I} \{ \rho'_{j} \}$ on which we define $d^Z$ as equal to $d$ on $\mathfrak X\times \mathfrak X$ and equal to $d^{(I)}$ on $\mathfrak X^{(I)} \times \mathfrak X^{(I)}$ which appears as the image of the first $I+1$ spaces $\mathfrak X_k$ together with the singletons added further. Both distances match when restricted to these images of the $I+1$ first spaces, and we define $d^Z$ accordingly.For the remaining cases, when $x \in \mathfrak X$ and $y \in \mathfrak X^{(I)}$ we set
\begin{equation*}
d^Z(x,y) := \inf\limits_{x',y'} \{ d(x,x') + d^{(I)}(y,y') + d^Z(x',y')\}
\end{equation*}
where the points $x',y'$ are taken in the image of the first $I+1$ spaces $\mathfrak X_i$ on which we already defined $d^Z$. On this space $Z$ we can compare the random variables $\zeta$ and $\zeta^{(I)}$ (more precisely their image via the isometries from $\mathfrak X$ and $\mathfrak X^{(I)}$ towards $Z$). We immediately have $d^Z(\zeta,\zeta^{(I)})= 0$ whenever $J>I$. Hence, again conditionally on $\mathfrak X,\mathfrak X^{(I)}$, we have $\mathbb P \left(d^Z(\zeta,\zeta^{(I)})>\varepsilon\right) \leq \sum_{i>I}s_i$. For any $\varepsilon>0$.

Now to bound Prohorov (and Gromov-Prohorov) distances, we make several uses of a corollary from Strassen \cite{strassen1965existence} page 438 which allows us to bound such distances by a coupling argument. Let $\varepsilon>0$, for $I$ large enough such that $\sum_{i>I} \leq \varepsilon$ (remember that $\mathbf s$ is fixed here) we have from what precedes $\mathbb P \left(d^Z(\zeta,\zeta^{(I)})>\varepsilon\right) \leq \varepsilon$. Hence the almost sure bound $d_{\gp} (\mu,\mu^{(I)})\leq \varepsilon$ by Strassen's theorem.

Using this theorem one more time with the coupling $\mathfrak X$, $\mathfrak X^{(I)}$ of distributions $\psi(Y)$ and $\psi^{(I)}(Y)$, we obtain the desired bound $d_{\mathrm P} (\psi^{(I)}(Y),\psi(Y)\leq \varepsilon$, which concludes the proof.
\end{proof}

\end{document}